\documentclass[a4paper]{article}

\usepackage{amsfonts, amsmath, amsthm, empheq, amssymb, bm, color,indentfirst}

\allowdisplaybreaks[4]

\newtheorem{thm}{Theorem}[section]
\newtheorem{lem}[thm]{Lemma}

\newtheorem{coro}[thm]{Corollary}
\newtheorem{defi}[thm]{Definition}

\newtheorem{prob}[thm]{Problem}
\newtheorem{rem}{Remark}[section]

\renewenvironment{abstract}{%
        \small
        \quotation
         \noindent {\bfseries \abstractname } }%

\def\f{\frac}
\def\d{\mathrm d}
\def\e{\mathrm e}
\def\i{\mathrm i}

\def\R{\mathbb R}
\def\N{\mathbb N}
\def\C{\mathbb C}
\def\A{\mathcal{A}}

\def\de{\delta}

\def\al{\alpha}
\def\be{\beta}
\def\ga{\gamma}
\def\Ga{\Gamma}
\def\ep{\epsilon}
\def\ve{\varepsilon}
\def\te{\theta}
\def\la{\lambda}

\def\na{\nabla}

\def\Om{\Omega}
\def\e{\mathrm{e}}
\def\d{\mathrm{d}}

\def\ov{\overline}
\def\pa{\partial}
\def\Re{\mathrm{\,Re\,}}
\def\Im{\mathrm{\,Im\,}}
\def\wh{\widehat}
\def\wt{\widetilde}
\def\LL{{L^2(\Om)}}
\def\L{\mathcal{L}}
\def\HH{{H^2(\Om)}}

\numberwithin{equation}{section}

\title{\Large\bf\boldmath
Initial-boundary value problems for multi-term time-fractional diffusion equations with $x$-dependent coefficients}

\author{
\large Zhiyuan LI$^\dag$\qquad Xinchi HUANG$^\ddag$ \qquad Masahiro YAMAMOTO$^\ddag$}

\date{}

\begin{document}
\maketitle

\renewcommand{\thefootnote}{\fnsymbol{footnote}}
\footnotetext{\hspace*{-5mm} 
\begin{tabular}{@{}r@{}p{14cm}@{}} 
& Manuscript last updated: \today.\\
$^\dag$ 
& School of Mathematics and Statistics, Shandong University of Technology, Zibo, Shandong 255049, China
E-mail: zyli@sdut.edu.cn\\
$^\ddag$ 
& Graduate School of Mathematical Sciences, 
the University of Tokyo,
3-8-1 Komaba, Meguro-ku, Tokyo 153-8914, Japan. 
E-mail: huangxc@ms.u-tokyo.ac.jp, 
myama@ms.u-tokyo.ac.jp.
\end{tabular}}

 
\begin{abstract}
In this paper, we discuss an initial-boundary value problem (IBVP) for the multi-term time-fractional diffusion equation with $x$-dependent coefficients. By means of the Mittag-Leffler functions and the eigenfunction expansion, we reduce the IBVP to an equivalent integral equation to show the unique existence and the analyticity of the solution for the equation. Especially, in the case where all the coefficients of the time-fractional derivatives are non-negative, by the Laplace and inversion Laplace transforms, it turns out that the decay rate of the solution for long time is dominated by the lowest order of the time-fractional derivatives. Finally, as an application of the analyticity of the solution, the uniqueness of an inverse problem in determining the fractional orders in the multi-term time-fractional diffusion equations from one interior point observation is established.
\end{abstract}


\section{Introduction}
\label{sec-intro}

In recent decades, more experimental data from diffusion processes in highly heterogeneous media indicate anomalous phenomena which cannot be described by the classical diffusion model with integer order derivative, for example, Adams and Gelhar \cite{AG92} points out that field data in a saturated zone of a highly heterogeneous aquifer indicate a long-tailed profile in the spatial distribution of densities as the time passes, which is difficult to be interpreted by the Gaussian processes. We also refer to Benson, Wheatcraft and Meerschaert \cite{BWM00} and Levy and Berkowitz \cite{LB03} for the anomalous phenomenon such as the non-Fickian growth rates and skewness exhibited in the solute concentration profiles. The above anomalous phenomenon has been investigated by many researchers, and see Berkowitz, Scher and Silliman \cite{BSS00}, Giona, Cerbelli and Roman \cite{GGR92}, Y. Hatano and N. Hatano \cite{HH98}, and the time-fractional diffusion equation: $\partial_t^\alpha u = \Delta u + F$, $(x,t)\in\mathbb R^d\times(0,\infty)$, $\alpha\in(0,1)$, is used. 
Here the Caputo derivative $\partial_t^\alpha$ is defined by 
$$
\partial_t^\alpha g(t) = \frac{1}{\Gamma(1-\alpha)} 
\int_0^t (t-\tau)^{-\alpha}\frac{d g(\tau)}{d\tau} d\tau
$$
(e.g., Kilbas, Srivastave and Trujillo \cite{KST06} and Podlubny \cite{P99}).
Here and henceforth, $\Gamma(\cdot)$ denotes the Gamma function. This model is presented as a useful approach for the description of transport dynamics in complex systems that are governed by anomalous diffusion and non-exponential relaxation patterns, and has attracted great attention in different areas. Theoretical researches have been developing rapidly, and here we refer only to a part of references: Gorenflo, Luchko and Yamamoto \cite{GLY15}, Gorenflo, Luchko and Zabrejko \cite{GLZ99}, Kubica and Yamamoto \cite{KY17}, Luchko \cite{L10},
Luchko and Gorenflo \cite{LG99}, Metzler and Klafter \cite{MK00}, Roman and Alemany \cite{RA94}, Sakamoto and Yamamoto \cite{SY11}, Xu, Cheng and Yamamoto \cite{XCY11} and Zacher \cite{Z09}.

Some recent publications such as e.g., Chechkin, Gorenflo and Sokolov \cite{CGS02}, Kubica and  Ryszewska \cite{KR17}, Kochubei \cite{K08} and Li, Luchko and Yamamoto \cite{LiLuYa14} investigate the time-fractional diffusion equations of distributed order derivative which is an integral of fractional derivatives with respect to continuously changing orders. Here we should mention an important particular case of the time-fractional diffusion equation of distributed order, that is, the weight function is taken in the form of a finite linear combination of the Dirac $\delta$-functions with the non-negative weight coefficients. This yields a diffusion equation with multiple time-fractional derivatives, which is the main focus of this paper. We deal with the following  initial-boundary value problem for the multi-term time-fractional diffusion equation:
\begin{equation}
\label{equ-multifrac}
\left\{ 
\begin{alignedat}{2}
&\sum_{j=1}^\ell q_j(x) \partial_t^{\alpha_j} u
=-\mathcal A u + B(x)\cdot\nabla u +b(x)u,
&\quad &(x,t)\in \Omega\times(0,T), 
\\
&u(x,0)=a(x), &\quad &x\in \Omega,
\\
&u(x,t)=0, &\quad &(x,t)\in \partial\Omega\times(0,T),
\end{alignedat}
\right.
\end{equation}
where $\Omega$ is a bounded domain in $\mathbb R^d$ with sufficiently smooth boundary $\partial\Omega$, for example, of $C^2$-class, and we assume $q_1=1$, $0<\alpha_\ell <\dots<\alpha_1<1$.  $-\mathcal A$ is a symmetric uniformly elliptic operator with the homogeneous Dirichlet boundary condition defined for $u\in H^2(\Omega) \cap H_0^1(\Omega)$:
$$
 (-\mathcal A u)(x)
= \sum_{i,j=1}^d \frac{\partial}{\partial x_i}\left(a_{ij}(x)
\frac{\partial}{\partial x_j}u(x)\right),
\quad x\in \Omega,
$$where $a_{ij}=a_{ji}\in C^1(\overline\Omega)$, $1\le i,j\le d$. Moreover there exists a constant $\nu>0$ such that
$$
\nu \sum_{j=1}^d \xi_j^2 
\le \sum_{j,k=1}^d a_{jk}(x) \xi_j\xi_k, 
\quad x\in\overline{\Omega},\ \xi\in\mathbb R^d.
$$

The mathematical investigation of the multi-term time-fractional diffusion equation has been already appeared in many publications. Here we limit ourselves to a few references and we do not intend a comprehensive list of the related works. In Daftardar-Gejji and Bhalekar \cite{DGB08}, a solution to an IBVP is formally represented by Fourier series and the multivariate Mittag-Leffler function. However no proofs for the convergence of the series are given in \cite{DGB08}. A proof of the convergence of the series defining the solution of the multi-term time-fractional diffusion equation with positive constant coefficients can be found in the paper Li, Liu and Yamamoto \cite{LiLiuYa14}. Jiang, Liu, Turner and Burrage \cite{JLTB12} discusses the case where the spatial dimension is one, the coefficients are constants and the spatial fractional derivative is considered, and establishes the formula of the solution. Luchko \cite{L11} proves the unique existence of the classical solution, the maximum principle and related properties in the case where the coefficients of the time derivatives are positive and independent of $x$, and the arguments are based on the Fourier method, that is, the separation of the variables. It reveals that these papers mainly discuss the case where the spatial differential operators are symmetric elliptic operators and the coefficients of time-fractional derivatives are positive constants. 

In this paper, we continue the research activities initiated in \cite{DGB08, 
JLTB12, LiLiuYa14, L11}  and investigate non-symmetric diffusion equations 
with the variable coefficients of fractional time derivatives which can be 
regarded as more feasible model equation in modeling diffusion in highly 
heterogeneous media with convection. Firstly, we discuss the unique existence as well as regularity of the solution to the IBVP \eqref{equ-multifrac}, which should be the starting points for further researches concerning the theory of non-linear fractional diffusion equations, numerical analysis and control theory. Secondly, we investigate some further properties of the solution including the analyticity and long-time asymptotic behavior of the solution. The asymptotic behavior of solutions to the equations, which describes some physical processes, is important both by itself and for analysis of the suitable numerical methods for their solutions and the inverse problems for these equations. For example, as a byproduct of the properties of the solution, we can show the uniqueness of an inverse problem of the determination of the fractional orders from one interior point observation. 

The rest of this article is organized as follows: In Section \ref{sec-thms}, by means of the Mittag-Leffler functions and the eigenfunction expansion, a formal integral equation of the solution of our problem \eqref{equ-multifrac} is first constructed, from which we further introduce a definition of the mild solution. The unique existence as well as the analyticity of the mild solution is proved in Section \ref{sec-wellposed}, whereas the long-time asymptotic formulas are given in Section \ref{sec-asymp}. In Section \ref{sec-IP}, we will give a proof of the uniqueness for an inverse problem of determining the fractional orders from one interior point observation. Finally, the last section is devoted to some conclusions and the statements of open problems.

 \section{Mild solution and main results}
 \label{sec-thms}
In this section, we solve the IBVP \eqref{equ-multifrac}. To this end, we start with fixing some general settings and notations. Let $L^2(\Omega)$ be a usual $L^2$-space with the inner product $(\cdot,\cdot)$, and $H^k(\Omega)$, $H_0^1(\Omega)$ denote usual Sobolev spaces (e.g., Adams \cite{A75}). We set $\|a\|_{L^2(\Omega)}=(a,a)_{L^2(\Omega)}^{\frac12}$. We define an operator $A$ in $L^2(\Omega)$ by
$$
(Af)(x)=(\mathcal A f)(x),\ x\in\Omega,\quad f\in D(A):=H^2(\Omega) \cap H_0^1(\Omega).
$$
Then the fractional power $A^\gamma$ is defined for $\gamma\in\mathbb R$ (e.g., Pazy \cite{P83}), and $D(A^\gamma)\subset H^{2\gamma}(\Omega)$, $D(A^{\frac 1 2})=H_0^1(\Omega)$ for example. Since $A$ is a symmetric uniformly elliptic operator, the spectrum of $A$ is entirely composed of eigenvalues and counting according to the multiplicities, we can set $0<\lambda_1 \le \lambda_2 \le \dots$. By $\phi_n \in D(A)$, we denote an orthonormal eigenfunction corresponding to $\lambda_n$: $A\phi_n=\lambda_n\phi_n$. Then the sequence $\{\phi_n\}_{n\in\mathbb N}$ is orthonormal basis in $L^2(\Omega)$. Moreover, we see that 
$$
A^\gamma \phi = \sum_{n=1}^\infty \lambda_n^\gamma (\phi,\phi_n) \phi_n,
$$
where
$$
\phi\in D(A^\gamma) 
:= \left\{\psi\in L^2(\Omega): \sum_{n=1}^\infty \lambda_n^{2\gamma} |(\psi,\phi_n)|^2<\infty\right\} 
$$
and that $D(A^{\gamma})$ is a Hilbert space with the norm
$$
\|\psi\|_{D(A^{\gamma})}
=\left( \sum_{n=1}^\infty \lambda_n^{2\gamma} |(\psi,\phi_n)|^2 \right)^{\frac 1 2}.
$$
Moreover we define the Mittag-Leffler function by
$$
E_{\alpha,\beta}(z):=\sum_{k=0}^\infty \frac{z^k}{\Gamma(\alpha k+\beta)},\ z\in\mathbb C,
$$
where $\alpha,\beta>0$ are arbitrary constants. The above formula and the classical asymptotics
\begin{equation}
\label{esti-Gamma}
\Gamma(\eta) 
= e^{-\eta}\eta^{\eta-\frac 1 2}(2\pi)^{\frac 1 2}
\left(1 + O\left(\frac{1}{\eta}\right)\right) 
\quad \mbox{as $\eta\rightarrow +\infty$}
\end{equation}
(e.g., Abramowitz and Stegun \cite{AS72}, p.257) imply that the radius of convergence is $\infty$ and so $E_{\alpha,\beta}(z)$ is an entire function of $z\in\mathbb C$. Furthermore, the following useful lemma holds:
\begin{lem}
\label{lem-ML}
{\rm(i)} 
Let $0<\alpha<2$ and $\beta>0$ be arbitrary. We suppose that $\frac{\pi}{2}\alpha<\mu<\min\{\pi,\pi\alpha\}$. Then there exists a constant $C= C(\alpha,\beta, \mu)>0$ such that
$$
|E_{\alpha,\beta}(z)| \le \frac{C}{1+|z|},
\quad \mu\le |\arg z|\le\pi.
$$
{\rm(ii)} 
For $\lambda>0$, $\alpha>0$ and positive integer $n\in\mathbb N$, we have
$$
\frac{d^n}{dt^n}E_{\alpha,1}(-\lambda t^\alpha)
= -\lambda t^{\alpha-n} E_{\alpha,\alpha-n+1}(-\lambda t^\alpha),
\quad t>0. 
$$
Moreover, $E_{\alpha,1}(-\lambda t^\alpha)$ with $0 < \alpha < 1$ is completely monotonic, that is, $(-1)^n\frac{d^n}{d t^n} E_{\alpha,1}(-\lambda t^\alpha)\ge 0$ for all $t > 0$ and $n = 0,1,\dots$.
\end{lem}
The proof of (i) can be found in Gorenflo and Mainardi \cite{GM97}, on p. 35 in Podlubny \cite{P99}. By the series expansion of the Mittag-Leffler function $E_{\alpha,\beta}$, the termwise differentiation yields (ii).
 
Now we define an operator $S(z):L^2(\Omega)\rightarrow L^2(\Omega)$  for $z\in\{z\in\C\setminus\{0\};\, |\arg z|<\frac{\pi}{2}\}$ by
\begin{align}
\label{def-S(t)}
S(z)a:=\sum_{n=1}^\infty (a,\phi_n) E_{\alpha_1,1}(-\lambda_nz^{\alpha_1}) \phi_n,
\quad a\in L^2(\Omega).
\end{align}
In view of (ii) in Lemma \ref{lem-ML}, the termwise differentiation gives
\begin{align*}
S^{(j)}(z)a
&:=-\sum_{n=1}^\infty \lambda_n(a,\phi_n) z^{\alpha_1-j}
        E_{\alpha_1,\alpha_1-j+1}(-\lambda_n z^{\alpha_1})\phi_n,
\quad j=1,2
\end{align*}
for $a\in L^2(\Omega)$, where $S^{(j)}(a)$ stands for $\frac{d^jS(z)a}{dz^j}$. We also adopt the abbreviation $S'(z)a:=\frac{dS(z)a}{dz}$ and $S''(z)a:=\frac{d^2S(z)a}{dz^2}$ for later use. Moreover, by Lemma 2.1 (i), we can prove that there exists a constant $C>0$ such that 
\begin{equation}
\label{esti-S(t)}
\|A^{\gamma-1}S^{(j)}(z)\|_{L^2(\Omega)\to L^2(\Omega)}
\leq C |z|^{\alpha_1-j-\alpha_1 \gamma}, 
\quad j=0,1,2
\end{equation}
for $z\in\{z\in\C\setminus\{0\};\, |\arg z|<\frac{\pi}2\}$ and $0\le\gamma\le1$, where $\|\cdot\|_{L^2(\Omega)\to L^2(\Omega)}$ denotes the operator norm from $L^2(\Omega)$ to $L^2(\Omega)$. 

Next we give the definition of the mild solution to \eqref{equ-multifrac}. As for mild solution for parabolic equation, see Pazy \cite{P83}, and here we need more arguments for fractional time-derivatives as follows. For this, we formally show an integral equation which is equivalent to \eqref{equ-multifrac}, which is only composed of $u,\nabla u$ without the time derivative of the solution. Indeed, from Sakamoto and Yamamoto \cite{SY11}, by regarding $-\sum_{j=2}^\ell q_j\partial_t^{\alpha_j} u+ B\cdot \nabla u + b u$  as non-homogeneous term in \eqref{equ-multifrac}, we have
$$
u(t) 
=S(t)a-\int^t_0 A^{-1}S'(t-s)(B\cdot \nabla u(s)
  + b u(s)) \d s
+ \sum^\ell_{j=2} \int^t_0 A^{-1}S'(t-s) q_j\partial_t^{\alpha_j} u(s) \d s.
$$
We consider the last term
$$
\sum_{j=2}^\ell \int^t_0 A^{-1}S'(t-s) q_j \partial_t^{\alpha_j} u(s) \d s.
$$
From the definition of Caputo fractional derivative, we have  
\begin{align*}
&\int^t_0 A^{-1}S'(t-s) 
    \Big(q_j\partial_t^{\alpha_j}u(s)\Big) \d s
=\int_0^t A^{-1}S'(t-s) \frac{1}{\Gamma(1-\alpha_j)}
\left(\int_0^s (s-r)^{-\alpha_j} q_j u'(r) dr\right) \d s,
\end{align*}
where we denote $u'(t):=\frac{d u}{d t}(t)$. By Fubini's theorem we exchange the orders of integrals and change the variable $s\to\xi$ by $\xi:=\frac{s-r}{t-r}$ to obtain
\begin{align*}
&\int^t_0 A^{-1}S'(t-s) q_j\partial_t^{\alpha_j}u(s) \d s 
= \int_0^t \left(\int_r^t A^{-1}S'(t-s) 
     \frac{(s-r)^{-\alpha_j}}{\Gamma(1-\alpha_j)} \d s \right) 
     q_j u'(r) \d r\\
=& \int^t_0\left(
           \int^1_0 A^{-1}S'\big((1-\xi)(t-r)\big) \frac{\xi^{-\alpha_j}}{\Gamma(1-\alpha_j)} \d\xi\right)
             (t-r)^{1-\alpha_j} q_j u'(r) \d r=:\frac{I(t)}{\Gamma(1-\alpha_j)}.
\end{align*}
Since the integrands have singularities at $\xi=0, 1$ and $r=t$, we should understand that 
\begin{align*}
I(t) &= \lim_{\ep_1, \ep_2, \ep_3 \downarrow 0}
\int_0^{t-\ep_3} \Big(\int^{1-\ep_1}_{\ep_2} 
A^{-1}S'\big((1-\xi)(t-r)\big) \xi^{-\al_j}  \d\xi \Big) 
(t-r)^{1-\al_j} q_j u'(r)\d r
\\
&=:\lim_{\ep_1, \ep_2, \ep_3 \downarrow 0}I_{\ep_1,\ep_2,\ep_3}(t).
\end{align*}
For computing the limit of $I_{\ep_1,\ep_2,\ep_3}(t)$ as $\ep_1, \ep_2, \ep_3 \downarrow 0$, we need take some further treatment on $I_{\ep_1,\ep_2,\ep_3}(t)$.  For this, integration by parts yields
\begin{align*}
I_{\ep_1,\ep_2,\ep_3}(t)
=& \left(\int_{\ep_2}^{1-\ep_1} A^{-1}S'\big((1-\xi)(t-r)\big) \xi^{-\al_j}\d\xi \right) 
      q_j u(r)(t-r)^{1-\al_j} 
\Big|_{r=0}^{r=t-\ep_3}\\
&+ \int_0^{t-\ep_3}\left(\int_{\ep_2}^{1-\ep_1} A^{-1}S''\big((1-\xi)(t-r)\big)
(1-\xi) \xi^{-\al_j}  \d\xi \right) (t-r)^{1-\al_j}
 q_ju(r)\d r\\
&+ \int_0^{t-\ep_3}\left(\int_{\ep_2}^{1-\ep_1} A^{-1}S'\big((1-\xi)(t-r)\big) \xi^{-\al_j}\d\xi\right)
   (1-\al_j) (t-r)^{-\al_j} q_j u(r)\d r
\\
=&: \sum_{k=1}^3 I_{\ep_1,\ep_2,\ep_3}^{(k)}(t).
\end{align*}
We evaluate each of the above three terms separately. First for $I_{\ep_1,\ep_2,\ep_3}^{(1)}(t)$, we conclude from \eqref{esti-S(t)} that
\begin{align*}
&\left\|
\int_{\ep_2}^{1-\ep_1} A^{-1}S'\big((1-\xi)(t-r)\big) \xi^{-\al_j} \d\xi
\right\|_{\LL\to\LL}
\le C\int_{\ep_2}^{1-\ep_1} \big((1-\xi)(t-r)\big)^{\al_1-1} \xi^{-\al_j}\d\xi.
\end{align*}
Moreover from the property of the Beta function that
\begin{equation}
\label{equ-beta}
\int_0^1 (1-\xi)^{\al-1} \xi^{\be-1} \d\xi  
=\f{\Ga(\al)\Ga(\be)} {\Ga(\al+\be)}<\infty,\ \al,\be>0,
\end{equation}
by $\al_1 > \al_j$ ($j=2,..., \ell$), for $r=t-\ep_3$ we have
\begin{align*}
&\left\|
\int_{\ep_2}^{1-\ep_1} A^{-1}S'((1-\xi)(t-r))
\xi^{-\al_j} \d\xi q_ju(r)(t-r)^{1-\al_j} 
\right\|_\LL
\le  C\ep_3^{\al_1-\al_j} \|u\|_{L^\infty(0,T;\LL)}
\to 0
\end{align*}
as $\ep_3 \to 0$.  Hence by $u(0) = a$, we see that
$$
\lim_{\ep_1,\ep_2, \ep_3\downarrow 0} I_{\ep_1,\ep_2,\ep_3}^{(1)}(t) 
= t^{1-\al_j} \int_0^1 A^{-1}S'((1-\xi)t)  \xi^{-\al_j} q_ja \d\xi.
$$ 
Next we estimate $I_{\ep_1,\ep_2,\ep_3}^{(2)}(t)$, again by using \eqref{esti-S(t)}, it follows that
$$
\|I_{\ep_1,\ep_2,\ep_3}^{(2)}(t) \|_\LL 
\le \int_0^{t-\ep_3} \left(\int_{\ep_2}^{1-\ep_1} (1-\xi)^{\al_1-1} \xi^{-\al_j}\d\xi\right)
(t-r)^{\al_1-\al_j-1}\d r \|u\|_{L^\infty(0,T;\LL)},
$$
the integrand is integrable in $0 < \xi < 1$ and $0 < r <t$ in view of \eqref{equ-beta} and we take the limit as $\ep_1,\ep_2,\ep_3 \downarrow 0$ to derive
$$
\lim_{\ep_1,\ep_2,\ep_3 \downarrow 0} I_{\ep_1,\ep_2,\ep_3 }^{(2)}(t)
=\int_0^t \left(\int_0^1 A^{-1}S''\big((1-\xi)(t-r)\big)
(1-\xi) \xi^{-\al_j}  \d\xi \right) (t-r)^{1-\al_j}
 q_ju(r)\d r.
$$
Finally, for $I_{\ep_1,\ep_2,\ep_3}^{(3)}(t)$ we argue similarly to obtain
\begin{align*}
& \lim_{\ep_1,\ep_2, \ep_3\downarrow 0} I_{\ep_1,\ep_2,\ep_3}^{(3)}(t) 
= (1-\al_j) \int^t_0 \left( \int^1_0 A^{-1}S'((1-\xi)(t-r)) \xi^{-\al_j} \d\xi\right)
(t-r)^{-\al_j} q_j u(r)\d r.
\end{align*}
Thus 
\begin{align*}
I(t)=&t^{1-\al_j} \int_0^1 A^{-1}S'((1-\xi)t)  \xi^{-\al_j} q_ja \d\xi\\
 &+  \int^t_0\left(\int^1_0 A^{-1}S''((1-\xi)(t-r)) (1-\xi)\xi^{-\al_j}\d\xi\right)
    (t-r)^{1-\al_j} q_ju(r)\d r\\
&+ (1-\al_j) \int^t_0 \left( \int^1_0 A^{-1}S'((1-\xi)(t-r)) \xi^{-\al_j} \d\xi\right)
(t-r)^{-\al_j} q_j u(r)\d r.
\end{align*}
Consequently we have
\begin{align}
\label{equ-int-u}
u(t) 
&=S(t)a+ \int_0^t A^{-1}S'(t-r)  \sum_{j=2}^\ell\frac{r^{-\alpha_j} q_ja}{\Gamma(1-\alpha_j)}\d r
  -\int^t_0 A^{-1}S'(t-r)(B\cdot \nabla u(r) + bu(r)) \d r
\nonumber\\
&+\int^t_0\int^1_0 A^{-1}S''\big((1-s)(t-r)\big)(1-s) \sum_{j=2}^\ell \frac{(t-r)^{1-\alpha_j} s^{-\alpha_j}q_ju(r)  }{\Gamma(1-\alpha_j)} \d s \d r
\nonumber\\
&+ \int^t_0 \int^1_0 A^{-1}S'\big((1-s)(t-r)\big) \sum_{j=2}^\ell \frac{(1-\alpha_j)(t-r)^{-\alpha_j} s^{-\alpha_j} q_j u(r)}{\Gamma(1-\alpha_j)} \d s \d r
=:\sum_{j=1}^5 I_j.
\end{align}

Now on the basis of the above integral equation of the solution $u$ which solves the IBVP \eqref{equ-multifrac}, we are ready to give the definition of the mild solution to \eqref{equ-multifrac}.
\begin{defi}[Mild solution]
\label{def-u-homo}
Let  $a\in L^2(\Omega)$, we call $u\in C([0,T];L^2(\Omega))\cap C((0,T];H_0^1(\Omega))$  a mild solution to the initial-boundary value problem \eqref{equ-multifrac} if it satisfies the integral equation \eqref{equ-int-u}.
\end{defi}
The above definition of the mild solution is well-defined because of our first main result:
\begin{thm}
\label{thm-homo}
Let $0<\alpha_\ell<\dots<\alpha_1<1$ and $T>0$ be fixed constants. Assuming that $q_j\in W^{2,\infty}(\Omega)$ $(j=2,\dots,\ell)$, $B\in (L^{\infty}(\Omega))^d$ and $b\in L^\infty(\Omega)$. Then for any fixed constant $\gamma\in[\frac12,1)$, the initial-boundary value problem \eqref{equ-multifrac} with $a\in L^2(\Omega)$ admits a unique mild solution $u\in C([0,T];L^2(\Omega))\cap C((0,T]; H^{2\gamma}(\Omega)\cap H_0^1(\Omega))$  such that
$$
\|u(t)\|_{H^{2\gamma}(\Omega)} \le Ct^{-\alpha_1\gamma}e^{CT} \|a\|_{L^2(\Omega)},\quad t\in(0,T].
$$
Moreover, $u:(0,T]\rightarrow H^{2\gamma}(\Omega)\cap H_0^1(\Omega)$ can be analytically extended to the sector $\{ z\in \mathbb C\setminus \{0\};|\arg z|<\frac{\pi}{2} \}$. 
\end{thm}
Here and henceforth, $C>0$ denotes constants which are independent of $t$, $T$, $a$ and $u$, but may depend on $\gamma$, $\{\alpha_j\}_{j=1}^\ell$, $d$, $b$, $B$, $\Omega$, $\{q_j\}_{j=2}^\ell$ and the coefficients of $\mathcal A$.
\begin{rem}
In Beckers and Yamamoto \cite{BY13} a similar fractional diffusion equation is discussed for $B=0$ and a similar regularity is proved. However \cite{BY13} assumes an extra condition $\alpha_1+\alpha_\ell > 1$, and our main result needs not such an assumption. 
\end{rem}
\begin{rem}
This theorem only deals with the homogeneous diffusion equation with multiple time-fractional derivatives. For the non-homogeneous case, one can refer to Jiang, Li, Liu and Yamamoto \cite{JLLY17}.
\end{rem}

Theorem \ref{thm-homo} shows that the spatial regularity $H^{2\gamma}(\Omega)$ cannot achieve the maximal $H^2(\Omega)$-regularity, but yields the continuity of the solution with respect to time $t\in(0,T]$ for 
$\frac{1}{2} < \gamma < 1$. 
However, the next theorem demonstrates that 
the solution $u(t)$ can achieve the maximal spatial regularity, $u(t)\in H^2(\Omega)$ for almost all $t\in(0,T)$.
\begin{thm}[$H^2(\Omega)$-regularity]
\label{thm-H2}
Let $0<\alpha_\ell<\dots<\alpha_1<1$ and $T>0$ be given.  Assuming that $a\in L^2(\Omega)$, $q_j\in W^{2,\infty}(\Omega)$ $(j=2,\dots,\ell)$, $B\in (L^{\infty}(\Omega))^d$ and $b\in L^{\infty}(\Omega)$. Then the solution $u$ to the initial-boundary value problem \eqref{equ-multifrac} belongs to $L^p(0,T;H_0^1(\Omega)\cap H^2(\Omega))$ with $1\leq p<\min\{2,\tfrac{1}{\alpha_1}\}$. Moreover the following estimate
$$
\|u\|_{L^p(0,T;H^2(\Omega))} \le C_T\|a\|_{L^2(\Omega)}
$$
holds true.
\end{thm}
\begin{rem}
This is different from the case of parabolic equations whose solutions cannot be in $L^p(0,T;H^2(\Omega))$ for any $p\ge1$ provided that an initial 
value is in $L^2(\Omega)$.
\end{rem}

If we further assume that $B\equiv0$ and all the coefficients $q_j$ are non-negative, then we have the following long-time asymptotic behavior of the solution to the IBVP \eqref{equ-multifrac}.
\begin{thm}
\label{thm-asymp}
Let $\alpha_j\in(0,1)$ be constants such that $\alpha_\ell<\dots<\alpha_1$, and $\{q_j\}_{j=1}^\ell$ be in $W^{2,\infty}(\Omega)$ with $q_j\ge,\not\equiv0$ on $\overline\Omega$ $(j=2,...,\ell)$. We further assume that $B\equiv0$ and $b\in 
W^{1,\infty}(\Omega)$ with $b\le0$ in $\Omega$. 
Let $v$ be a unique solution to 
the initial-boundary value problem 
\begin{equation}
\label{equ-single}
\left\{\begin{alignedat}{2}
&q_\ell(x)\partial_t^{\alpha_\ell}v(x,t)
=-\mathcal A v(x,t)+b(x)v(x,t), &\quad& x\in\Omega,\ t>0,
\\
&v(x,0)=a(x),  &\quad&  x\in\Omega,
\\
&v(x,t)=0,  &\quad&  x\in\partial\Omega,\ t>0.
\end{alignedat}
\right.
\end{equation}
Then the solution $u$ to \eqref{equ-multifrac} has the same asymptotic behavior as $v$, in the sense that
$$
\|u(\cdot\,,t) - v(\cdot\,,t)\|_{H^2(\Omega)}
= O(t^{-\min\{2\alpha_\ell,\alpha_{\ell-1}\}})\|a\|_{L^2(\Omega)},\quad
\mbox{as } t\to\infty.
$$
\end{thm}
Theorem \ref{thm-asymp} shows that as $t \to \infty$,
the solution $u$ to the IBVP \eqref{equ-multifrac} tends to the solution $v$ to the IBVP \eqref{equ-single} with a single time-fractional derivative. The assumption $b\leq0$ and $q_j\ge,\not\equiv0$ on $\overline\Omega$ are necessary for proving that the Laplace transform $\widehat{u}(x,s)$ of the solution $u$ to our problem \eqref{equ-multifrac} has no poles in the main sheet of Riemann surface cutting off the negative axis, which is essential for the proof of Theorem \ref{thm-asymp}. In the case of negative coefficients $\{q_j\}$, a counterexample can be found in \cite{LiLiuYa14}. Moreover, from this theorem, we can see that the decay rate of $u$ is $t^{-\alpha_\ell}$ which is the best possible one. More precisely, we have the following statement:
\begin{coro}
\label{cor-asymp}
Under the same assumptions in Theorem \ref{thm-asymp}, we have the following estimate
$$
\left\|u(\,\cdot\,,t) 
- \frac{(\mathcal A-b)^{-1}(q_\ell a) t^{-\alpha_\ell}}{\Gamma(1-\alpha_\ell)} \right\|_{H^2(\Omega)}
\leq C\|a\|_{L^2(\Omega)} t^{-\min\{2\alpha_\ell, \alpha_{\ell-1}\}},
$$
for sufficiently large $t>0$, where $(\mathcal A-b)^{-1}(q_\ell a)$ denotes the unique solution of $(\mathcal A-b)w = q_\ell a$ for $w\in H^2(\Omega)\cap H_0^1(\Omega)$, and the constant $C>0$ is independent of $t$, $a$ and $u$, but may depend on $d$, $\Omega$, $\{\alpha_j\}_{j=1}^\ell$, $\{q_j\}_{j=1}^\ell$, $b$ and $\{a_{ij}\}_{i,j=1}^d$.

Moreover, suppose that $\|u(\cdot\,,t)\|_{H^2(\Omega)}=o(t^{-\alpha_\ell})$ as $t\to\infty$, then $u(x,t)=0$ for all $x\in\Omega$ and $t>0$.
\end{coro}

From the above theorems, it turns out that the fractional orders are very related to the asymptotic behavior of the solutions of the time-fractional diffusion equations. Regarding the practical importance and theoretical interests, we propose the following inverse problem
\begin{prob}
\label{prob-IP}
Let $T>0$ be arbitrarily given and $x_0\in\Om$ be any fixed point. Assume that the initial value $a\in \LL$ with $a\not\equiv0$ of fixed sign $($i.e. $a\ge,\not\equiv0$ or $a\le,\not\equiv0$$)$. We let $u$ satisfy the IBVP \eqref{equ-multifrac}. Determine $\al_j$ $(0\le j\le\ell)$ by the one interior point observation $u(x_0,t)$, $t\in(0,T)$.
\end{prob}
We can refer to papers on inverse problems of the determination of the fractional orders in the time-fractional diffusion models. For the case of a single time-fractional derivative, Cheng, Nakagawa, Yamamoto and Yamazaki \cite{C09} proved the uniqueness for determining the fractional order and the diffusion coefficient by one endpoint measurement in the one-dimensional case, and on the basis of asymptotic behavior of the solution, Hatano, Nakagawa, Wang and Yamamoto \cite{H13} established a formula of reconstructing the order of fractional derivative in time in the fractional diffusion equation by time history at one fixed interior point. The uniqueness for the recovery of the fractional orders was proved for the multi-term case in Li and Yamamoto \cite{LY14}. See Kian, Oksanen, Soccorsi and Yamamoto \cite{KOSY18}, and Li, Imanuvilov and Yamamoto \cite{LIY15} on Dirichlet-to-Neumann map for fractional equation. Kian, Soccorsi and Yamamoto \cite{KSY17} investigated diffusion equations with time-fractional derivatives of space-dependent variable order and proved that the space-dependent variable order coefficient is uniquely determined by the knowledge of a suitable time-sequence of partial Dirichlet-to-Neumann maps. We refer to Li, Luchko and Yamamoto \cite{LLY16} for the inverse problem in the determination of the weight function in the diffusion model with distributed order time-derivatives. We refer to Li, Zhang, Jia and Yamamoto \cite{L13} for the numerical treatment, and Jin and Rundell \cite{J15} for a topical review and a comprehensive list of bibliographies. 

To the authors' knowledge, it reveals that the existing papers treat the inverse problems from single measurement in the case where all the coefficients $q_j$ in \eqref{equ-multifrac} are constants. Keeping the above points in mind, we are interested in the inverse problem of the determination of the fractional order in the fractional model \eqref{equ-multifrac}. We have
\begin{thm}[Uniqueness]
\label{thm-ip-multi}
Under the same assumptions in Theorem \ref{thm-asymp}, we further assume that $1\le d\le3$ and $a\not\equiv0$ is of fixed sign in $\Omega$. Moreover, we suppose that $u,\wt u\in C([0,T];\LL)\cap C((0,T]; \HH\cap H_0^1(\Omega))$ solve the IBVP \eqref{equ-multifrac} with respect to the fractional orders $0<\alpha_\ell<\dots<\alpha_1<1$ and $0<\wt\alpha_\ell<\dots<\wt\alpha_1<1$ separately. Then $\al_j=\wt\al_j, j=1,\dots,\ell$, if $u(x_0,\cdot)=\wt u(x_0,\cdot)$ in $(0,T)$.
\end{thm}

\section{Forward problem}
\label{sec-forward}
In this section, we mainly investigate the well-posedness for the IBVP \eqref{equ-multifrac} which will be divided into two subsections. In Section \ref{sec-wellposed}, we will finish the proof of Theorem \ref{thm-homo}, that is, the 
well-posedness of the problem \eqref{equ-multifrac} and the analyticity of the solution. In Section \ref{sec-asymp}, by Section \ref{sec-wellposed} and the Laplace transform argument, the long-time asymptotic behavior in Theorem \ref{thm-asymp} is easily proved.

\subsection{Unique existence and analyticity of the mild solution}
\label{sec-wellposed}
Since the mild solution to the IBVP \eqref{equ-multifrac} satisfies the integral equation \eqref{equ-int-u}, after the change of variables, we find
\begin{align}
\label{equ-int-u-standard}
u(t) 
=&S(t)a
  -\sum_{j=2}^\ell\frac{t^{1-\alpha_j}}{\Gamma(1-\alpha_j)}
      \int_0^1 A^{-1}S'(rt) (1-r)^{-\alpha_j}q_ja \d r
  -t\int_0^1 A^{-1}S'(rt) \widetilde u\big((1-r)t\big) \d r
\nonumber\\
&+\sum_{j=2}^\ell \frac{(1-\alpha_j)t^{1-\alpha_j}}{\Gamma(1-\alpha_j)}
       \int_0^1 \int_0^1 A^{-1}S'\big((1-s)rt\big) r^{-\alpha_j} s^{-\alpha_j} q_ju\big((1-r)t\big) \d s \d r
\nonumber\\
&+ \sum_{j=2}^\ell \frac{t^{2-\alpha_j}}{\Gamma(1-\alpha_j)}
        \int_0^1 \int_0^1 A^{-1}S''\big((1-s)rt\big) (1-s)r^{1-\alpha_j} s^{-\alpha_j} 
         q_ju\big((1-r)t\big) \d s \d r,
\end{align}
where $\widetilde u :=B\cdot\nabla u+bu$. Moreover, we extend the variable $t$ in \eqref{equ-int-u-standard} from $(0,T)$ to the sector $\{z\in \mathbb C\setminus\{0\}; |\arg z|<\frac{\pi}{2}\}$, and setting $u_0=0$, we define $u_{n+1}(z)$ $(n=0,1,\dots)$ as follows:
\begin{align} 
\label{def-u_n} 
u_{n+1}(z)
=&S(z)a
  -\sum_{j=2}^\ell\f{z^{1-\al_j}}{\Ga(1-\al_j)}\int_0^1 A^{-1}S'(rz)(1-r)^{-\al_j}q_ja\d r
\nonumber\\
 -&z\int_0^1 A^{-1}S'(rz) (B\cdot\na u_n + bu_n)\big((1-r)z\big)\d r
\nonumber\\
+&\sum_{j=2}^\ell \f{(1-\al_j)z^{1-\al_j}}{\Ga(1-\al_j)}
       \int_0^1 \int_0^1 A^{-1}S'\big((1-s)rz\big) r^{-\al_j} s^{-\al_j} q_ju_n\big((1-r)z\big)\d s\d r
\nonumber\\
+ &\sum_{j=2}^\ell \f{z^{2-\al_j}}{\Ga(1-\al_j)}
        \int_0^1 \int_0^1 A^{-1}S''\big((1-s)rz\big) (1-s)r^{1-\al_j} s^{-\al_j} 
         q_ju_n\big((1-r)z\big)\d s\d r.
\end{align}
We conclude from the definition \eqref{def-S(t)} of $S(z)$ and the properties of Mittag-Leffler function that $u_n(z)$ defined in \eqref{def-u_n} uniformly converges to the solution to the initial-boundary value problem \eqref{equ-multifrac} as $n\to\infty$ for any compact subset of the section $\{z\in\mathbb C\setminus\{0\};\, |\arg z|<\frac{\pi}{2}\}$. The details are listed as follows.
\begin{proof}[\bf Proof of Theorem \ref{thm-homo} ]
For any $n\in\N$, taking the operator $A^\ga$ on both sides of \eqref{def-u_n}, from \eqref{esti-S(t)} for the $z\in S_{\te,T}:=\{z\in\mathbb C\setminus\{0\};\, |\arg z|<\te, |z|\le T\}$  with $\te\in(0,\f{\pi}{2})$, we claim that the following estimate holds:
\begin{align}
\label{esti-u_n}
\|u_{n+1}(z)-u_n(z)\|_{D(A^\ga)}
\leq M_1M^n \left( \sum_{j=1}^\ell J^{\be_j} \right)^n \big(g\big)(|z|)\|a\|_\LL,
\quad n\in\N,
\end{align}
where $g(t):=t^{-\al_1\ga}$, $\be_1:=\al_1-\al_1\ga$, $\be_j:=\al_1-\al_j$, $j=2,\dots,\ell$, the constant $M$ is independent of $T$, $t>0$, $z\in S_{\te,T}$, but may dependent on $\ga$, $d$, $\Om$, $\te$, $p$, $p_1, ..., p_\ell$, $\al_1, ..., \al_\ell$, and by $J^\al$ we denote the Riemann-Liouville fractional integral
$$
(J^\al f)(t):=\f{1}{\Ga(\al)}\int_0^t (t-\tau)^{\al-1} f(\tau) \d\tau,\quad \al>0,
$$
and denote $J^0f(t) = f(t)$. We now proceed by induction on $n$ to prove the inequality \eqref{esti-u_n}.  Firstly, for $n=0$, using the estimate \eqref{esti-S(t)}, for $z\in S_{\te,T}$, it follows that
\begin{align*}
\|u_{1}(z) - u_0(z)\|_{D(A^\ga)}
\le& \sum_{j=2}^\ell 
    \left\|
    \f{-z^{1-\al_j}}{\Ga(1-\al_j)}\int_0^1 A^{\ga-1}S'(rz)(1-r)^{-\al_j}q_ja\d r 
    \right\|_\LL +\|A^\ga S(z)a\|_\LL
\\
\le& C\left(\sum_{j=2}^\ell |z|^{\al_1-\al_j-\al_1\ga}
        \int_0^1 r^{\al_1-1-\al_1\ga}(1-r)^{-\al_j}\d r
     +|z|^{-\al_1\ga} \right)\|a\|_\LL.
\end{align*}
Since $\ga\in[\f 1 2,1)$, and noting that $|z|^{\al_1-\al_j-\al_1\ga}\leq T^{\al_1-\al_j}|z|^{-\al_1\ga}$ ($j=1,\dots,\ell$), we see that
$$
\|u_{1}(z) - u_0(z)\|_{D(A^\ga)}
\le C\sum_{j=1}^\ell T^{\al_1-\al_j}|z|^{-\al_1\ga}\|a\|_\LL
=:M_1 |z|^{-\al_1\ga}\|a\|_\LL.
$$
Next, for any $n\in\N$, in view of the inequalities  $\| B\cdot\na v\|_\LL \le C\| v\|_{D(A^{\f 1 2})} \le C\|v\|_{D(A^\ga)}$  for $v \in D(A^\ga)$ and $\ga\in[\f 1 2,1)$, we derive
\begin{align*}
 &\| A^{-1}S'(r z)\big( B\cdot\na (u_{n+1}-u_n) + b(u_{n+1}-u_n)\big)\big((1-r)z\big)\|_{D(A^\ga)}
 \\
\le& C\| A^{\ga-1}S'(r z)\|_{\LL\to \LL}
 \|(u_n-u_{n-1})((1-r)z)\|_{D(A^\ga)}.
\end{align*}
Combining the above inequalities with \eqref{esti-S(t)} for $z\in S_{\te,T}$, we can prove that
\begin{align*}
&\|u_{n+1}(z)- u_n(z)\|_{D(A^\ga)}
\\
\le& C|z|^{\be_1}
      \int_0^1 (1-r)^{\be_1-1}
        \|u_n(rz)-u_{n-1}(rz)\|_{D(A^{\gamma})}\d r
\\
    &+C\sum_{j=2}^\ell|z|^{\be_j}
      \left(\int_0^1(1-s)^{\be_1-1} s^{-\al_j}\d s \right)
        \int_0^1 (1-r)^{\be_j-1}  \|u_n(rz)-u_{n-1}(rz)\|_{D(A^\ga)}\d r,
\end{align*}
where $\be_1:=\al_1-\al_1\ga$, $\be_j:=\al_1-\al_j$, $j=2,\dots,\ell$. Noting that, for $0 < \al_j < \al_1<1$, $j=2,\dots,\ell$  and \eqref{equ-beta}, we have
$$
\|u_{n+1}(z)-u_n(z)\|_{D(A^\ga)} 
\leq C\sum_{j=1}^\ell |z|^{\be_j} \int_0^1 (1-r)^{\be_j-1}
        \|u_n(rz)-u_{n-1}(rz)\|_{D(A^\ga)}\d r.
$$
Consequently, by inductive assumption, we can prove
$$
\|u_{n+1}(z)- u_n(z)\|_{D(A^\ga)}
\leq C\sum_{j=1}^\ell |z|^{\be_j} 
      \int_0^1 (1-r)^{\be_j-1} M_1M^{n-1}
         \left(\sum_{i=1}^\ell J^{\be_i}\right)^{n-1}\big(g\big)(r|z|) \d r.
$$
After making the change of variable $r\to r|z|$ and from the definition of the Riemann-Liouville fractional integral, we see that
\begin{align*}
&\|u_{n+1}(z)- u_n(z)\|_{D(A^\ga)}
\leq CM_1M^{n-1}\sum_{j=1}^\ell 
   \int_0^{|z|} (|z|-r)^{\be_j-1} 
      \left(\sum_{i=1}^\ell J^{\be_i}\right)^{n-1}\big(g\big)(r) \d r
\\
=& CM_1M^{n-1}\sum_{j=1}^\ell 
     \Ga(\be_j) J^{\be_j}  
        \left( \left(\sum_{i=1}^\ell J^{\be_i}\right)^{n-1}\big(g\big)\right) (|z|)
\\
\le &CM_1M^{n-1}\max_{1\le j\le\ell}\{\Ga(\be_j)\}
     \sum_{j=1}^\ell J^{\be_j} 
        \left( \left(\sum_{i=1}^\ell J^{\be_i}\right)^{n-1}\big(g\big)\right) (|z|)
=M_1M^n \left(\sum_{i=1}^\ell J^{\be_i}\right)^n\big(g\big)(|z|),
\end{align*}
where we set $M:=C\max_{1\le j\le\ell}\{\Ga(\be_j)\}$. Therefore by indcution, \eqref{esti-u_n} holds true. Moreover, noting the semigroup property
$$
J^\al J^\be = J^{\al+\be},\quad \al\ge0,\ \be\ge0,
$$
and the effect of the operator $J^\al$ on the power functions
$$
J^\al t^\be 
= \f{\Ga(\be+1)}{\Ga(\be+1+\al)} t^{\al+\be},
\quad \al\ge0,\ \be>-1,\ t>0,
$$
we derive
\begin{align}
\label{esti-u_n-T}
&\|u_{n+1}(z)- u_n(z)\|_{D(A^\ga)}
\leq M_1M^n 
\sum_{k_1+\dots+k_\ell=n}  
  \binom{n}{k_1}\dots \binom{n}{k_\ell}  
   J^{\be_1k_1+\dots+\be_\ell k_\ell}(g)(|z|)
\nonumber\\
=&M_1M^n 
\sum_{k_1+\dots+k_\ell=n}  
  \binom{n}{k_1}\dots \binom{n}{k_\ell} 
  \f{\Ga(1-\al_1\ga)|z|^{\be_1k_1+\dots+\be_\ell k_\ell-\al_1\ga}}
       {\Ga(\be_1k_1+\dots+\be_\ell k_\ell+1-\al_1\ga)},
\quad z\in S_{\te,T}.
\end{align}
Here we noticed for any subset $K$ compacted in $S_{\te,T}$ that
$$
\sum_{n=0}^{\infty} M^n 
\sum_{k_1+\dots+k_\ell=n}  
  \binom{n}{k_1}\dots \binom{n}{k_\ell} 
  \f{\Ga(1-\al_1\ga)|z|^{\be_1k_1+\dots+\be_\ell k_\ell-\al_1\ga}}
       {\Ga(\be_1k_1+\dots+\be_\ell k_\ell-\al_1\ga+1)}
$$
converges uniformly in $K$. In fact, the asymptotic behavior \eqref{esti-Gamma} yields
$$
\Ga(\be_1k_1+\dots+\be_\ell k_\ell-\al_1\ga+1)
\ge C\Ga(\underline{\be}(k_1+\dots+k_\ell)-\al_1\ga+1)
=C\Ga(\underline{\be}n-\al_1\ga+1),
$$
and noting that $\sum_{k_1+\dots+k_\ell=n}\binom{n}{k_1}\dots\binom{n}{k_\ell}
=\ell^n$, it follows for $z\in S_{\te,T}$ that
$$
\sum_{n=0}^{\infty} 
\sum_{k_1+\dots+k_\ell=n}  
  \binom{n}{k_1}\dots \binom{n}{k_\ell} 
  \f{ M^n|z|^{\be_1k_1+\dots+\be_\ell k_\ell-\al_1\ga}}
       {\Ga(\be_1k_1+\dots+\be_\ell k_\ell-\al_1\ga+1)}
\leq C\sum_{n=0}^{\infty} 
        \ell^n \f{M^n T^{\ov{\be}n}|z|^{-\al_1\ga}}
                    {\Ga(\underline{\be}n-\al_1\ga+1)},
$$
where $\ov{\be}:=\max_{1\le j\le\ell}\{\be_j\}$, $\underline{\be}:=\min_{1\le j\le\ell}\{\be_j\}$. Again using the asymptotic behavior (\ref{esti-Gamma}), we find
$$
\f{\ell^{n+1}M^{n+1}T^{\ov \be (n+1)}}
     {\Ga(\underline{\be}(n+1)-\al_1\ga+1)}
\Big/ \f{\ell^n M^n T^{\ov \be n}}
       {\Ga(\underline \be n-\al_1\ga+1)}
\longrightarrow 0 \quad \mbox{as $n \to \infty$},
$$
so that 
$$
\sum_{n=0}^\infty
  M^n \sum_{k_1+\dots+k_\ell=n}  
        \binom{n}{k_1}\dots \binom{n}{k_\ell} 
        \f{|z|^{\be_1k_1+\dots+\be_\ell k_\ell-\al_1\ga}}
             {\Ga(\be_1k_1+\dots+\be_\ell k_\ell-\al_1\ga+1)}
< \infty.
$$ 
Hence the majorant test implies $\sum_{n=1}^\infty \|u_{n+1}(z) - u_n(z)\|_{D(A^\ga)}$ is convergent uniformly in any compact subset of $S_{\te,T}$. Therefore there exists $u_*(z)\in D(A^\ga)$ such that $\|u_n(z) - u_*(z)\|_{D(A^\ga)}$ tends to $0$ as $n \rightarrow \infty$ uniformly in any compact subset of $S_{\te,T}$. We thus assert that $u=u_*|_{\Om\times(0,T]}$ is the unique solution to the integral equation \eqref{equ-int-u}.

Furthermore, we can see from \eqref{esti-u_n} that $\|u_*(t)\|_{D(A^\ga)} =O(\e^{Ct})$, as $t\to\infty$. Indeed, for any $T\ge1$ and $0< t\le T$, we have
\begin{align*}
&\| A^\ga u_*(t)\|_\LL
\leq \sum_{n=0}^\infty
       \| A^\ga u_{n+1}(t)-A^\ga u_n(t)\|_\LL
\\
\leq& M_1\sum_{n=0}^{\infty} 
        M^n \sum_{k_1+\dots+k_\ell=n}  
              \binom{n}{k_1}\dots \binom{n}{k_\ell} 
              \f{\Ga(1-\al_1\ga)T^{\be_1k_1+\dots+\be_\ell k_\ell-\al_1\ga}}
                   {\Ga(\be_1k_1+\dots+\be_\ell k_\ell-\al_1\ga+1)}
=:M_1\Ga(1-\al_1\ga)H(T).
\end{align*}
The estimate of $H(t)$ as $t\to\infty$ follows from the fact that the Laplace transform
\begin{align*}
\L H(s)
:=&\int_0^\infty \sum_{n=0}^{\infty}  
    \sum_{k_1+\dots+k_\ell=n}  \binom{n}{k_1}\dots \binom{n}{k_\ell} 
      \f{M^n t^{\be_1k_1+\dots+\be_\ell k_\ell-\al_1\ga}}
           {\Ga(\be_1k_1+\dots+\be_\ell k_\ell-\al_1\ga+1)}\e^{-st} \d t
\\
 =&\f{s^{\al_1\ga-1}}{1-M\sum_{j=1}^\ell s^{-\be_j}},
\end{align*}
where $\Re s>M_2$ and $M_2>0$ is a sufficiently large constant, has only finite simple poles in the main sheet of Riemann surface cutting off the negative axis. We denote the poles as $\{s_1,\dots,s_m\}\subset\C$. Moreover, we can see that $s_i\in\R$ and $s_i>0$, $i=1,\dots,m$. Indeed, for $s:=r\e^{\i\te}$ with $\te\in[-\pi,\pi]$ such that $1-M\sum_{j=1}^\ell s^{-\be_j}=0$, that is, 
$$
   \sum_{j=1}^\ell r^{-\be_j}\cos\be_j\te
-\i\sum_{j=1}^\ell r^{-\be_j}\sin\be_j\te
 =\f 1 M,
$$
which implies $\sum_{j=1}^\ell r^{-\be_j}\sin\be_j\te=0$, and noting that $\sin\be_j\te$ ($j=1,\dots,\ell$) have the same signals, hence $\te=0$. Now by Fourier-Mellin formula (e.g., \cite{S99}), we have
$$
H(t)=\f{1}{2\pi \i}\int_{M_2-\i\infty}^{M_2+\i\infty} \L H(s)\e^{st} \d s.
$$
Now we choose a small constant $0<\ga<\min\{s_1,\dots,s_m\}$ such that $1-M\sum_{j=1}^\ell s^{-\be_j}\ne 0$ for $s\ne s_1,\dots,s_m$, $\Re s\ge\ga$, and then by Residue Theorem (e.g., \cite{R87}), for $t>0$ we see that
$$
H(t)
=\sum_{j=1}^m a_{j} \e^{s_jt} 
+\f{1}{2\pi\i}\int_{\ga-\i\infty}^{\ga+\i\infty} \L H(s)\e^{st} \d s,
$$
where $a_{j}:=\lim_{s\to s_j}(s-s_j)\L H(s)$, and the shift in the line of integration is justified by the fact $e^{st}\L H(s)\to 0$ as $\Im\, s\to\infty$ with $\Re s$ bounded.

Integration by parts shows
\begin{align*}
&\int_{\ga-\i\infty}^{\ga+\i\infty} \L H(s)\e^{st} \d s
= \f{s^{\al_1\ga-1}}{1-M\sum_{j=1}^\ell s^{-\be_j}}
   \f{\e^{st}}{t}\bigg|_{s=\ga-\i\infty}^{s=\ga+\i\infty} \\
&\qquad\qquad\quad-\int_{\ga-\i\infty}^{\ga+\i\infty} 
     \f{\e^{st}}{t} \f{(\al_1\ga-1)s^{\al_1\ga-2}(1-M\sum_{j=1}^\ell s^{-\be_j})
                  -s^{\al_1\ga-1}M\sum_{j=1}^\ell \be_js^{-\be_j-1}}
                {(1-M\sum_{j=1}^\ell s^{-\be_j})^2} \d s.
\end{align*}
Therefore
$$
\left|H(t)-\sum_{j=1}^m a_{j} \e^{s_jt}\right| 
= O\left(\f 1 t \e^{\ga t}\right), \quad t\to\infty.
$$
See e.g., \cite{H81} as for a similar argument. Consequently, we derive that 
$$
H(t)=O(\e^{Ct}),\quad t\to\infty,
$$
which implies
\begin{equation}
\label{esti-u*}
\|u_*(t)\|_{D(A^{\ga})}
\leq Ct^{-\al_1\ga}\e^{CT}\|a\|_\LL,
\quad t\in(0,T].
\end{equation}

Let us turn to show the analyticity of the solution with respect to $z$. For this, by induction, we first prove that $u_n: S_{\te,T}\to D(A^\ga)$ is analytic for $n=0,1,\dots$. By $u_0\equiv0$, it is obvious for $n=0$. We inductively assume that $u_n: S_{\te,T}\to D(A^\ga)$ is analytic in $z$. We estimate the integrands in \eqref{def-u_n}. The use of \eqref{esti-S(t)} implies
\begin{align*}
&\|A^{-1} S'(rz)(1-r)^{-\al_j} q_ja\|_{D(A^\ga)} 
\\
\le& C(rz)^{\al_1-\al_1\ga-1} (1-r)^{-\al_j} \|a\|_\LL
\\
\le& C|z|^{\al_1-\al_1\ga-1} r^{\al_1-\al_1\ga-1} (1-r)^{-\al_j} \|a\|_\LL,
\end{align*}
\begin{align*}
&\|A^{-1} S'(rz) (B\cdot\na u_n + bu_n)((1-r)z))\|_{D(A^\ga)} 
\\
\le& C(rz)^{\al_1-\al_1\ga-1} \|u_n((1-r)z)\|_{D(A^{\f 1 2})}
\\
\le& C|z|^{\al_1-\al_1\ga-1} r^{\al_1-\al_1\ga-1} \|u_n((1-r)z)\|_{D(A^{\f 1 2})},
\end{align*}
\begin{align*}
&\|A^{-1} S'((1-s)rz)r^{-\al_j}s^{-\al_j} q_ju_n((1-r)z)\|_{D(A^\ga)}
\\
\le & C((1-s)rz)^{\al_1-1} r^{-\al_j}s^{-\al_j} \|u_n((1-r)z)\|_{D(A^\ga)}
\\
\le& C|z|^{\al_1-1} r^{\al_1-\al_j-1} (1-s)^{-\al_j} s^{-\al_j} \|u_n((1-r)z)\|_{D(A^\ga)}
\end{align*}
and
\begin{align*}
&\|A^{-1} S''((1-s)rz)(1-s)r^{1-\al_j}s^{-\al_j} q_ju_n((1-r)z)\|_{D(A^\ga)}
\\
\le & C((1-s)rz)^{\al_1-2} (1-s)r^{1-\al_j}s^{-\al_j} \|u_n((1-r)z)\|_{D(A^\ga)}
\\
\le& C|z|^{\al_1-2} r^{\al_1-\al_j-1} s^{-\al_j}(1-s)^{\al_1-1}  \|u_n((1-r)z)\|_{D(A^\ga)}.
\end{align*}
Moreover, in view of \eqref{esti-Gamma} and \eqref{esti-u_n-T}, it follows that 
$$
\|u_n(z)\|_{D(A^\ga)} 
\le C_n |z|^{-\al_1\ga} \|a\|_\LL,
\quad z\in S_{\te,T},
$$
hence that the $D(A^\ga)$-norm of the integrands in \eqref{def-u_n} are integrable in $r,s\in(0,1)$. Therefore $u_{n+1}: S_{\te,T}\to D(A^\ga)$ is analytic. Thus by induction, we see that $u_n: S_{\te,T}\to D(A^\ga)$ is analytic for all $n\in\N$. Since we have proved that $\sum_{n=1}^\infty \|u_{n+1}(z) - u_n(z)\|_{D(A^\ga)}$ is convergent 
uniformly in any compact subset of $S_{\te,T}$, therefore $u_*:S_{\te,T}\to D(A^\ga)$ is analytic. Moreover, since $T$ and $\te$ are arbitrarily chosen, we deduce $u_*$ is analytic in the sector  $\{s\in\C\setminus\{0\}; |\arg s|<\f{\pi}{2}\}$. 

Finally, we see that $u(\cdot,t)$ ($t\in(0,T]$) is just the solution to \eqref{equ-multifrac} and such that
$$
\|u(\cdot,t)\|_{H^{2\ga}(\Om)} \le Ct^{-\al_1\ga} e^{CT}\|a\|_\LL,
\quad \ga\in[\f 1 2,1),\ t\in(0,T]
$$
in view of \eqref{esti-u*}. This completes the proof of the theorem.
\end{proof}
\begin{rem}
\label{rem-regularity}
If we furthermore assume that 
$$
\mbox{$b\in W^{1,\infty}(\Om)$ or $b(\le0)\in L^\infty(\Om)$, and $B\in \{W^{1,\infty}(\Om)\}^d$,}
$$
we point out that the solution $u(t)$ to the initial-boundary value problem \eqref{equ-multifrac} can achieve more regularity on time and space, that is, $u\in C((0,T]; H_0^1(\Om)\cap \HH)$ and
$$
\|u(t)\|_\HH 
\le Ct^{-\al_1} e^{CT} \|a\|_\LL,
\quad 0<t\le T.
$$
\end{rem}

Now let us turn to the proof of Theorem \ref{thm-H2}.
\begin{proof}[\bf Proof of Theorem \ref{thm-H2} ]
Taking the operator $A$ on the both sides of \eqref{equ-int-u}, we evaluate each of the five terms separately. Estimate of $I_1(t)$. We conclude from \eqref{esti-S(t)} that
$$
\|I_1(t)\|_{D(A)}
\leq Ct^{-\al_1}\|a\|_\LL,\quad t\in(0,T].
$$
In order to estimate $\|I_2(t)\|_{D(A)}$, we break up the integral in $I_2$ into two integrals as follows
\begin{align*}
I_2(t)
&=\sum_{j=2}^\ell \int_0^{\f t 2} A^{-1}S'(r)(t-r)^{-\al_j}  \f{q_ja}{\Ga(1-\al_j)}\d r
 +\sum_{j=2}^\ell \int_{\f t 2}^t A^{-1}S'(r)(t-r)^{-\al_j}  \f{q_ja}{\Ga(1-\al_j)}\d r
\\
&=:I_{21}(t)+I_{22}(t).
\end{align*}
For $\|I_{21}(t)\|_{D(A)}$. Integrating by parts derives
\begin{align*}
AI_{21}(t)
=\sum_{j=2}^\ell \f{1}{\Ga(1-\al_j)}
   S(r)(t-r)^{-\al_j}q_ja \Big|_{r=0}^{r=\f t 2}
-\sum_{j=2}^\ell \f{\al_j}{\Ga(1-\al_j)} \int_0^{\f t 2}S(r)(t-r)^{-\al_j-1}  q_ja\d r.
\end{align*}
Since $\|S(t)\|_{\LL\to\LL}$ is uniformly bounded for $t\in[0,T]$, it is easily seen that 
$$
\|I_{21}(t)\|_{D(A)}\le C\sum_{j=2}^\ell t^{-\al_j}\|a\|_\LL,\quad t\in(0,T].
$$
For $\|I_{22}(t)\|_{D(A)}$, from \eqref{esti-S(t)}, it follows
$$
\|I_{22}(t)\|_{D(A)}
\le C\sum_{j=2}^\ell \int_{\f t 2}^t r^{-1}(t-r)^{-\al_j}\d r\|a\|_\LL
\le C\sum_{j=2}^\ell t^{-\al_j}\|a\|_\LL.
$$
For $I_3(t)$, $0<t\le T$, again from \eqref{esti-S(t)}, recalling the definition of $S(t)$ in \eqref{def-S(t)}, we derive
\begin{align*}
&\|I_3(t)\|_{D(A)}^2
=\left\| \int_0^t S'(t-r)(B\cdot\na u(r)+bu(r))\d r \right\|_\LL^2
\\
\le& \sum_{n=1}^\infty
     \left| \int_0^t \la_n(t-r)^{\al_1-1}
                E_{\al_1,\al_1}(-\la_n(t-r)^{\al_1})
         (B\cdot\na u(r)+bu(r),\phi_n)\d r \right|^2.
\end{align*}
Thus the Young inequality implies
$$
   \int_0^T\|I_3(t)\|_{D(A)}^2 \d t
   \leq \sum_{n=1}^\infty
     \Big(
     \int_0^T \la_n r^{\al_1-1}\left|E_{\al_1,\al_1}(-\la_n r^{\al_1})\right| \d r\Big)^2
     \int_0^T|(B\cdot\na u(r)+bu(r),\phi_n)|^2 \d r.
$$
Moreover, the use of Lemma \ref{lem-ML} derives that
$$
\int_0^T \left|\la_n r^{\al_1-1}E_{\al_1,\al_1}(-\la_n r^{\al_1})\right|\d r
=\int_0^T \la_n r^{\al_1-1}E_{\al_1,\al_1}(-\la_n r^{\al_1})\d r
=1-E_{\al_1,1}(-\la_nT^{\al_1})
\le C_T,
$$
thereby obtaining the inequalities
\begin{align*}
   \int_0^T\|I_3(t)\|_{D(A)}^2 \d t
   \le& C_T\int_0^T\sum_{n=1}^\infty |(B\cdot\na u(r)+bu(r),\phi_n)|^2\d r
   \le C_T\int_0^T\|u(t)\|_{H^1(\Om)}^2 \d t
\\
   \le& C_T\int_0^T (t^{-\f 1 2 \al_1}\|a\|_\LL)^2\d t
   \le C_T\|a\|_\LL^2.
\end{align*}
Here in the last inequality we used the estimate \eqref{esti-u*}. For $I_4(t)$, from \eqref{esti-S(t)}, select $\ep>0$ small enough so that $\al_1(1-\ep)>\al_j$ ($j=2,\dots,\ell$),  and similar to the argument used in Theorem \ref{thm-homo}, it follows that
\begin{align*}
\|I_4(t)\|_{D(A)} 
\le& C\left\|\sum_{j=2}^\ell \int^t_0\int^1_0 A^{\ep-1}S''\big((1-s)(t-r)\big)(1-s) (t-r)^{1-\alpha_j} s^{-\alpha_j} A^{1-\ep}q_ju(r)  \d s \d r \right\|_\LL
\\
\le& C\sum_{j=2}^\ell  \left(\int_0^1 (1-s)^{\al_1-\al_1\ep-1} s^{-\al_j}\d s\right)
          \int_0^t (t-r)^{\al_1(1-\ve)-1-\al_j} \|A^{1-\ve}u(r)\|_\LL \d r.
\end{align*}
Again the use of \eqref{esti-u*} leads to
$$
\|I_4(t)\|_{D(A)}\le C_T\sum_{j=2}^\ell t^{-\al_j}\|a\|_\LL,
\quad 0<t\le T.
$$
For $I_5(t)$ we argument similarly to obtain
$$
\|I_5(t)\|_{D(A)}
\le C_T\sum_{j=2}^\ell t^{-\al_j}\|a\|_\LL.
$$
Finally, we proved that for any $t\in(0,T]$, the solution $u$ satisfies
$$
\left(\int_0^T \|u(t)\|_\HH^p \d t\right)^{\f 1 p}
\le C_T \|a\|_\LL,\quad 1\le p<\min\{2,\tfrac{1}{\al_1}\}.
$$
This completes the proof of the theorem.
\end{proof}

\subsection{Long-time asymptotic behavior}
\label{sec-asymp}

In this part, we investigate the large time behavior of the solution $u$ to the IBVP \eqref{equ-multifrac} under the assumptions that $B\equiv0$ and $b\le0$. Based on the results in Section \ref{sec-wellposed} and Remark \ref{rem-regularity}, we know that the mild solution $u$ to the problem \eqref{equ-multifrac} uniquely exists in $H_0^1(\Omega)\cap H^2(\Omega)$ for any $t\in(0,T]$ and admits 
\begin{equation}
\label{esti-H2}
\|u(\cdot,t)\|_\HH
\le Ct^{-\al_1}e^{CT}\|a\|_\LL,
\quad t\in(0,T].
\end{equation}
Thus the asymptotic behavior near $0$ is only related to the largest order of the fractional derivatives. As for the long-time asymptotic behavior, for $\ell=1$ in \eqref{equ-multifrac}, Sakamoto and Yamamoto \cite{SY11} asserts that the solution decays in polynomial $t^{-\alpha_1}$ as $t\to\infty$, which is a typical property of fractional diffusion equations in contrast to the exponential decay in the classical diffusion equations. In Li et al. \cite{LiLiuYa14}, the IBVP \eqref{equ-multifrac} with positive-constant coefficients was investigated. The Laplace transform in time was applied to show that the decay rate is indeed $t^{-\alpha_\ell}$ at best as $t\to\infty$, where $\alpha_\ell$ is the lowest order of the time-fractional derivatives. 

Here in this section we are devoted to the long-time asymptotic behavior of the solution to the IBVP \eqref{equ-multifrac}, and attempt to establish results parallel to that for the case of positive-constant coefficients.

The key idea to prove the main result in this section is using Laplace and inversion Laplace transforms. From \eqref{esti-H2}, we can apply the Laplace transform $\wh{\cdot}$ on both sides on the equation in \eqref{equ-multifrac}, and use the formula
$$
\wh{\pa_t^\al f}(s)=s^\al\wh{f}(s)-s^{\al-1}f(0+),
$$
to derive the transformed algebraic equation
$$
(\A-b(x))\wh u(x;s)+ Q(x;s) \wh u(x;s)
= s^{-1}Q(x;s) a(x),\quad \Om\times\{\Re s>M\},
$$
where we set 
$$
Q(x;s):=\sum_{j=1}^\ell q_j(x) s^{\al_j}.
$$
We check at once that $\wh u:\{\Re s>M\}\to \HH$ is analytic, which is clear from the property of Laplace transform. Moreover, we claim that $\wh{u}(x,s)$ ($\Re s>M$) can be analytically extended to the sector 
$$
S_{\te}:=\{ s\in\C\setminus\{0\};\ |\arg\,s|<\te\},
\quad \tfrac{\pi}{2}<\te<\min\{\tfrac{\pi}{2\al_1},\pi\}.
$$
More precisely, the following lemma holds.
\begin{lem}
\label{lem-lap-u}
Under the assumptions in Theorem \ref{thm-homo}, the Laplace transform $\wh u$ of the unique mild solution $u$ to the initial-boundary value problem \eqref{equ-multifrac} can be analytically extended to the sector $S_\te$. Moreover, there exists a positive constant $C$ only depending on $d$, $\Om$, $\te$, $b$, $\{\al_j\}_{j=1}^\ell$, $\{q_j\}_{j=1}^\ell$, $\{a_{ij}\}_{i,j=1}^d$ such that
\begin{equation}
\label{esti-lap-u}
\|\wh u(\cdot\,;s)\|_{H^1(\Om)}
\leq C\sum_{j=1}^\ell r^{\al_j-1}\|a\|_\LL, \quad 
\forall s=r\e^{\i \rho}\in S_{\te}.
\end{equation}
\end{lem}
\begin{proof}
Firstly from Theorem \ref{thm-homo}, we see that the solution $u$ to the initial-boundary value problem \ref{equ-multifrac} can be analytically extended to the sector $S_{\frac{\pi}2}:=\{ s\in\C\setminus\{0\}; |\arg s|<\f{\pi}{2} \}$. Therefore by an argument similar to the proof of Theorem 0.1 in \cite{P93}, we can prove that the Laplace transform $\wh u: \{\Re s>M\}\to \HH$ can be analytically extended to the sector $S_\te$.

Now let us turn to give an estimate for the Laplace transform $\wh u: S_\te\to\HH$. For this, we define a bilinear operator $B[\Phi,\Psi;s]: H_0^1(\Om)\times H_0^1(\Om)\to\C$ by
$$
B[\Phi,\Psi;s]
:=\int_\Om\big((\A-b)\Phi(x)\big) \ov{\Psi(x)}+Q(x;s)\Phi(x)\ov{\Psi(x)} \d x,\quad 
s\in S_\te,
$$
where $\ov\Psi$ denotes conjugate of $\Psi$. Integration by parts yields
$$
B[\Phi,\Psi;s]
=\int_\Om\sum_{i,j=1}^d a_{ij}\pa_i \Phi(x) \ov{\pa_j\Psi(x)}+(Q(x;s)-b(x))\Phi(x)\ov{\Psi(x)} \d x.
$$
Taking $\Phi=\Psi$ implies
$$
B[\Phi,\Phi;s] 
=\int_\Om \sum_{i,j=1}^d a_{ij}\pa_i \Phi(x) \ov{\pa_j \Phi(x)}+(Q(x;s)-b(x))|\Phi(x)|^2 \d x,
$$
hence that
\begin{align*}
\Re (B[\Phi,\Phi;s]) 
=\int_\Om \sum_{i,j=1}^d a_{ij}
(\Re\pa_i\Phi(x) \Re\pa_j\Phi(x) +\Im\pa_i\Phi(x)\Im\pa_j\Phi(x))
\\+(\Re Q(x;s)-b(x))|\Phi(x)|^2 \d x.
\end{align*}
From $\Re Q(x;s)=\sum_{j=1}^\ell q_j(x)r^{\al_j}\cos\al_j\rho>0$ in view of $s=r\e^{\i\rho}\in S_\te$, and $b\le0$, it follows that
$$
\Re(B[\Phi,\Phi;s]) 
\geq\int_\Om \sum_{i,j=1}^d a_{ij} \big(\Re\pa_i\Phi(x) \Re\pa_j\Phi(x)
            + \Im\pa_i\Phi(x) \Im\pa_j\Phi(x) \big)\d x.
$$
The ellipticity of $\{a_{ij}\}$ and the use of Poincar\'e's inequality imply
$$
\Re(B[\Phi,\Phi;s]) 
\geq C \| \Re \na \Phi\|_\LL^2 + C\|\Im \na \Phi\|_\LL^2
\geq C\|\Phi\|_{H^1(\Om)}^2.
$$
Consequently
\begin{align*}
C\|\wh u(\cdot\,;s)\|_{H^1(\Om)}^2
\le& |B[\wh u(\cdot\,;s), \wh u(\cdot\,;s)\,;s] |
= \big|(s^{-1}Q(\cdot;s))a,\wh u(\cdot\,;s))_\LL\big|
\\
\le& C\sum_{j=1}^\ell |s|^{\al_j-1} \|a\|_\LL \|\wh u(\cdot\,;s)\|_{H^1(\Om)}
\end{align*}
in view of the H\"{o}lder inequality, finally that 
$$
\|\wh u(\cdot\,;s)\|_{H^1(\Om)}
\le C\sum_{j=1}^\ell |s|^{\al_j-1} \|a\|_\LL,
\mbox{ for $s\in S_\te$}.
$$
The proof of Lemma \ref{lem-lap-u} is completed.
\end{proof}

\begin{proof}[\bf Proof of Theorem \ref{thm-asymp}]
By Fourier-Mellin formula (e.g., \cite{S99}), we have
$$
u(x,t)=\f{1}{2\pi \i}\int_{M-\i\infty}^{M+\i\infty} \wh u(x;s)\e^{st} \d s.
$$
From Lemma \ref{lem-lap-u}, we see that the Laplace transform $\wh{u}(x;s)$ of the solution to the initial-boundary value problem \eqref{equ-multifrac} is analytic in the sector $S_\te$. Therefore by Residue Theorem (e.g., \cite{R87}), for $t>0$ we see that the inverse Laplace transform of $\wh u$ can be represented by an integral on the contour $\ga(\ep,\te)$ defined as  $\{s\in\C; \arg s=\te,\ |s|\ge\ep\}\cup \{s\in\C; |\arg s|\le\te,\ |s|= \ep\}$, that is 
$$
u(x,t)
=\f{1}{2\pi\i}\int_{\ga(\ep,\te)} \wh u(x;s)\e^{s t} \d s,
$$
where in fact the shift in the line of integration is justified by the estimate \eqref{esti-lap-u}. Moreover, again from the estimate \eqref{esti-lap-u}, we can let $\ep$ tend to $0$, then we have
$$
u(x,t)=\f{1}{2\pi\i}\int_{\ga(0,\te)} \wh{u}(x;s)\e^{s t} \d s.
$$
On the other hand, we can repeat the above argument used in Lemma \ref{lem-lap-u} to derive that $\wh v(\cdot\,;s)\in H_0^1(\Om)$,  where $v$ solves the problem \eqref{equ-single} and
\begin{equation}
\label{esti-lap-v}
\|\wh{v}(\cdot\,;s)\|_{H^1(\Om)}\le C|s|^{\al_\ell-1}\|a\|_\LL,
\mbox{ for $s\in\ga(0,\te)$,}
\end{equation}
hence 
$$
v(t)=\f{1}{2\pi\i}\int_{\ga(0,\te)} \wh v(x;s)\e^{s t} \d s.
$$
Thus
\begin{equation}
\label{esti-u-v}
\|u(\cdot\,,t) - v(\cdot\,,t)\|_\HH
\le C\int_{\ga(0,\te)}\|\wh u(\cdot\,;s)-\wh v(\cdot\,;s)\|_\HH |\e^{st} \d s|.
\end{equation}
Noting that $\wh u - \wh v$ satisfies the following problem
\begin{equation*}
\left\{
\begin{alignedat}{2}
&(\A - b)(\wh u-\wh v)
+q_\ell(x)s^{\al_\ell} (\wh u - \wh v) + \sum_{j=1}^{\ell-1} q_j(x) s^{\al_j} \wh u
= \sum_{j=1}^{\ell-1} q_j(x)s^{\al_j-1} a(x), &\quad  &x\in\Om,\ s\in S_{\te},
\\
&\wh u(x;s) - \wh v(x;s)=0,  &\quad  &x\in\pa\Om,\ s\in S_{\te}.
\end{alignedat}
\right.
\end{equation*}
Then using the boundary regularity estimates in elliptic equations combining the inequality \eqref{esti-lap-u}, we deduce that
\begin{align}
\label{esti-lap-u-v}
&\|\wh u (\cdot\,;s) - \wh v (\cdot\,;s)\|_\HH
\nonumber\\
\le& C|s|^{\al_\ell} \|\wh u(\cdot\,;s) - \wh v(\cdot\,;s)\|_\LL
+ C\left(\sum_{i=1}^\ell \sum_{j=1}^{\ell-1} |s|^{\al_i+\al_j-1} +\sum_{j=1}^{\ell-1} |s|^{\al_j-1}\right)
\|a\|_\LL.
\end{align}
Now for $0<\de_0<1$ small enough such that $C\de_0^{\al_\ell}\le\f 1 2$, we break up the integral in \eqref{esti-u-v} into two parts
\begin{align*}
\|u(\cdot\,,t) - v(\cdot\,,t)\|_\HH
\le& C\left(\int_0^{\de_0} + \int_{\de_0}^\infty\right)
       \|\wh u(\cdot\,;r\e^{\i\te})-\wh v (\cdot\,;r\e^{\i\te})\|_\HH \e^{rt\cos\te} \d r
\\
=&:I_1(t;\de_0) + I_2(t;\de_0).
\end{align*}
For $I_1(t;\de_0)$ $(t>0)$, we conclude from \eqref{esti-lap-u-v} and Poincar\'{e}'s inequality that
$$
\|\wh u(\cdot\,;s) - \wh v(\cdot\,;s)\|_\HH
\le 2C\Big(\sum_{i=1}^\ell \sum_{j=1}^{\ell-1} |s|^{\al_i+\al_j-1} +\sum_{j=1}^{\ell-1} |s|^{\al_j-1}\Big)
\|a\|_\LL,\quad |s|\le \de_0,
$$
which implies
$$
I_1(t;\de_0)
\le \int_0^{\de_0}\|\wh u(\cdot\,;r\e^{\i\te})-\wh v(\cdot\,;r\e^{\i\te})\|_\HH \e^{rt\cos\te} \d r
\le C\Big(\sum_{i=1}^\ell \sum_{j=1}^{\ell-1} t^{-\al_i-\al_j} 
+\sum_{j=1}^{\ell-1} t^{-\al_j}\Big) \|a\|_\LL.
$$
For $I_2(t;\de_0)$ $(t>0)$, the use of \eqref{esti-lap-u-v} yields
\begin{align*}
&\|\wh u(\cdot\,;s) - \wh v(\cdot\,;s)\|_\HH
\\
\le& C|s|^{\al_\ell} \Big(\|\wh u(\cdot\,;s)\|_\LL +\|\wh v(\cdot\,;s)\|_\LL\Big)
+ C\left(\sum_{i=1}^\ell
\sum_{j=1}^{\ell-1} |s|^{\al_i+\al_j-1} +\sum_{j=1}^{\ell-1} |s|^{\al_j-1}\right)\|a\|_\LL,
\end{align*}
where $|s|\geq \de_0$, hence combining \eqref{esti-lap-u} with \eqref{esti-lap-v} gives
$$
I_2(t;\de_0)\le C\left( \sum_{j=1}^\ell t^{-\al_\ell-\al_j} + t^{-2\al_\ell}
           +\sum_{i=1}^\ell\sum_{j=1}^{\ell-1} t^{-\al_i-\al_j} 
           +\sum_{j=1}^{\ell-1} t^{-\al_j}\right)\|a\|_\LL.
$$
Substituting the estimates for $I_1(t;\de_0)$ and $I_2(t;\de_0)$ into \eqref{esti-u-v}, we can assert that
$$
\|u(\cdot\,,t) - v(\cdot\,,t)\|_\HH \le Ct^{-\al} \|a\|_\LL, \mbox{ $t>0$ large enough.}
$$
where $\al:=\min\{2\al_\ell,\al_{\ell-1}\}$.  This completes the proof of Theorem \ref{thm-asymp}.
\end{proof}
\begin{proof}[\bf Proof of Corollary \ref{cor-asymp}]
In order to prove the asymptotic behavior of $u$, we denote $u_\ell=\f{(\A-b)^{-1}(q_\ell a)t^{-\al_\ell}}{\Ga(1-\al_\ell)}$  and notice that the Laplace transform $\wh u_\ell$ of $u_\ell$ is $(\A-b)^{-1}(q_\ell a)s^{\al_\ell-1}$ 
and satisfies 
$$
\A \wh u_\ell-b\wh u_\ell = q_\ell s^{\al_\ell-1} a,\quad \wh u_\ell(\cdot\,;s)\in H_0^1(\Om),\quad 
s\in S_\te.
$$ 
Thus $\wh v - \wh u_\ell$ satisfies
$$
\left\{
\begin{alignedat}{2}
&(\A- b)(\wh v(x;s) - \wh u_\ell(x;s)) 
= - q_\ell(x) s^{\al_\ell} \wh v(x;s), &\quad &x\in\Om,\ s\in S_{\te},
\\
&\wh v(x;s) - \wh{u}_\ell(x;s)=0,  &\quad  &x\in\pa\Om,\ s\in S_{\te},
\end{alignedat}
\right.
$$
Therefore, the regularity estimate for elliptic equations and \eqref{esti-lap-v} combined yield
$$
\|\wh v(\cdot\,;s) - \wh u_\ell(\cdot\,;s)\|_\HH
\le C|s|^{2\al_\ell-1}\|a\|_\LL.
$$
An argument similar to the proof in Theorem \ref{thm-asymp} implies
$$
\|v(\cdot\,,t) - u_\ell(\cdot\,,t)\|_\HH \le Ct^{-2\al_\ell} \|a\|_\LL,
$$
hence
$$
\|u(\cdot\,,t) - u_\ell(\cdot\,,t)\|_\HH 
\le Ct^{-\min\{2\al_\ell,\al_{\ell-1}\}} \|a\|_\LL,
\mbox{ for $t>0$ large enough,}
$$
which completes the proof of Corollary \ref{cor-asymp}.
\end{proof}


\section{Inverse problem}
\label{sec-IP}

In this section, we will give a proof of Theorem \ref{thm-ip-multi}. The basic idea is first to use the Laplace transform to transfer the time-fractional diffusion equation to the corresponding elliptic equation with the Laplacian parameter. Then from the strong maximum principle for the elliptic equation, we can finish the proof.

\begin{proof}[\bf Proof of Theorem \ref{thm-ip-multi}]

According to our assumptions, Theorem \ref{thm-homo} and Remark \ref{rem-regularity}, the solution $u: (0,T]\to H^2(\Omega)\cap H_0^1(\Omega)$ to the IBVP \eqref{equ-multifrac} can be analytically extended from $(0,T)$ to $(0,\infty)$, and by the same notation we denote the extension. Then we arrive at the following initial-boundary value problem
\begin{equation}
\label{equ-multifrac_infty}
\left\{ 
\begin{alignedat}{2}
&\sum_{j=1}^\ell q_j(x) \partial_t^{\alpha_j} u
=-\mathcal A u + b(x)u,
&\quad &(x,t)\in \Omega\times(0,\infty), 
\\
&u(x,0)=a(x), &\quad &x\in \Omega,
\\
&u(x,t)=0, &\quad &(x,t)\in \partial\Omega\times(0,\infty).
\end{alignedat}
\right.
\end{equation}
The same is also true for $\wt u$ in place of $u$. Then taking Laplace transforms $\wh\cdot(s)$ on both sides of the equation \eqref{equ-multifrac_infty} with respect to $\al_j$ and $\wt\al_j$, we find
$$
\left\{\begin{alignedat}{2}
& \sum_{j=1}^\ell q_j(x) s^{\al_j} \wh u(s) + (\mathcal A - b)\wh u(s)=\sum_{j=1}^\ell q_j(x) s^{\al_j-1}a, &\quad& x\in\Om,
\\
& \wh u(s)=0, &\quad& x\in\pa\Om,
\end{alignedat}\right.
$$
and 
$$
\left\{\begin{alignedat}{2}
& \sum_{j=1}^\ell q_j(x) s^{\wt\al_j} \wh{\wt u}(s) + (\mathcal A - b)\wh{\wt u}(s) = \sum_{j=1}^\ell q_j(x) s^{\wt\al_j-1}a, &\quad& x\in\Om,
\\
& \wh{\wt u}(s)=0, &\quad& x\in\pa\Om,
\end{alignedat}\right.
$$
for any $s>M$, where $M>0$ is sufficiently large constant. 

Let us start with some observations mainly about the properties of the Laplace transforms $\wh u(s)$ and $\wh{\wt u}(s)$. Firstly, noting the positivity of the coefficients $q_j$, $-b$, from the strong maximum principle for the elliptic equations, we see that $\wh u(s)$ and $\wh{\wt u}(s)$ are strictly positive in the domain $\Om$ if $a\ge0,\not\equiv0$, and $\wh u(s)$ and $\wh{\wt u}(s)$ are strictly negative in the case $a\le0,\not\equiv0$. Next, recalling the arguments in Section \ref{sec-asymp}, from the above elliptic equation with Laplacian parameter $s$, we can further analytically extend $\wh u(s)$ from the 
domain $\{s>M\}$ to $\{s>0\}$ and satisfies
\begin{equation}
\label{esti-u1}
\|\wh u(s)\|_{H^1(\Om)}\le C\|a\|_{L^2(\Om)} \sum_{j=1}^\ell s^{\al_j-1}, \quad s>0,
\end{equation}
which implies that
\begin{equation}
\label{esti-u2<u0}
\|s\wh u(s)\|_{H^1(\Om)}
\le C\|a\|_\LL \sum_{j=1}^\ell s^{\al_j}, \quad 
\|s\wh{\wt u}(s)\|_{H^1(\Om)}
\le C\|a\|_\LL \sum_{j=1}^\ell s^{\wt\al_j}, \quad s>0.
\end{equation}

Now by taking the difference of the above two systems, it turns out that the system for 
$v:=\wh u-\wh{\wt u}$ reads
\begin{equation}
\label{equ-v}
\left\{\begin{alignedat}{2}
& \sum_{j=1}^\ell q_j(x) s^{\al_j} v(s) + (\mathcal A -b)v(s)
=(a - s\wh{\wt u}(s))\sum_{j=1}^\ell q_j(x) (s^{\al_j-1}-s^{\wt\al_j-1}),
&\quad& x\in\Om, \\
& \wh v(s)=0, &\quad& x\in\pa\Om.
\end{alignedat}\right.
\end{equation}

We prove our theorem by contradiction. We assume that $(\al_\ell,\dots,\al_1)\ne(\wt\al_\ell,\dots,\wt\al_1)$ and by $j_0$ denote the largest index such that 
$\al_{j}\ne\wt\al_{j}$, that is, $\al_{j_0}\ne\wt\al_{j_0}$ and $\al_j=\wt\al_j$ for $j\ge j_0+1$.

Without loss of generality, we assume that $\al_{j_0}<\wt\al_{j_0}$. Then by dividing the equations in \eqref{equ-v} by $\sum_{j=1}^\ell \left|s^{\al_j-1}-s^{\wt\al_j-1}\right|$, we see that
\begin{equation}
\label{equ-w}
\left\{\begin{alignedat}{2}
& \sum_{j=1}^\ell q_j(x) s^{\al_j} w(s) + (\mathcal A -b)w(s)
=(a - s\wh{\wt u}(s))Q_1(s), 
&\quad& x\in\Om, \\
& w(s)=0, &\quad& x\in\pa\Om,
\end{alignedat}\right.
\end{equation}
where we set
$$
w(s):=\f{sv(s)}{\sum_{j=1}^\ell \left|s^{\al_j}-s^{\wt\al_j}\right|},\quad
Q_1(s):=\f{\sum_{j=1}^\ell q_j(x) (s^{\al_j}-s^{\wt\al_j})}{\sum_{j=1}^\ell \left|s^{\al_j}-s^{\wt\al_j}\right|}.
$$
From the continuity of the analytic functions, we assert that  the following estimate
$$
\sum_{j=1}^\ell |s^{\al_j}-s^{\wt\al_j}|
=s^{\al_{j_0}}\sum_{j=1}^{j_0} |s^{\al_j-\al_{j_0}}-s^{\wt\al_j-\al_{j_0}}|>0
$$
holds true for any $0<s<1$. Therefore, $w(s)$ is well-defined for any $0<s<1$. From the property of the Laplace transforms, it follows that $\wh u(s)$ and $\wh{\wt u}(s)$ are analytic with respect to $s>0$ so that $w(s)$ is continuous for $s\in(0,1)$ in view of the definition of $w$. Moreover, by employing an argument similar to the estimate \eqref{esti-u2<u0} to system \eqref{equ-v}, and noting that $\al_j=\wt\al_j$ in the case of $j>j_0$, we conclude that 
$$
\|sv(s)\|_{H^1(\Om)}
\le C\|s\wh{\wt u}(s) - a\|_\LL \sum_{j=1}^\ell |s^{\wt\al_{j}} - s^{\al_{j}}|
\le C\|a\|_\LL(C\sum_{j=1}^\ell s^{\wt\al_j} + 1)\sum_{j=1}^{j_0} |s^{\wt\al_{j}} - s^{\al_{j}}|,
$$
hence that

$$
\|w(s)\|_{H^1(\Om)}
= \f{\|sv(s)\|_{H^1(\Om)}}{\sum_{j=1}^{j_0} \left|s^{\wt\al_j}-s^{\al_j}\right|}
\le C\|a\|_\LL(C\sum_{j=1}^\ell s^{\wt\al_j} +1),
$$
for all $0<s<1$. Furthermore, the above inequality implies 
$$
\|w(s)\|_{H^1(\Om)}
\le C(s^{\wt\al_\ell} +1) \|a\|_\LL,
$$
for small $0<s<1$. By above inequality, we also have
\begin{equation}
\label{esti-sw}
\lim_{s\to0+} s^{\al_j}w(s) = 0,\quad j=1,\dots,\ell.
\end{equation}
Next we claim that the limit of $w(s)$ exists as $s\to0+$ and 
$$
\begin{cases}
\lim_{s\to0+}w(s)\ge0, &\mbox{if $a\ge0,\not\equiv0$,}\\
\lim_{s\to0+}w(s)\le0, &\mbox{if $a\le0,\not\equiv0$.}
\end{cases}
$$
For this, we first need to prove that 
\begin{equation}
\label{lim-Q}
\lim_{s\to0+}Q_1(x\,;s)=q_{j_0}(x), \quad \text{for any } x\in\Om.
\end{equation}
In fact, for the proof, we just need to discuss small $s\in (0,1)$. By our assumption $\al_{j_0}<\wt\al_{j_0}$, we have $s^{\al_{j_0}}-s^{\wt\al_{j_0}}>0$ for small $0<s<1$. Thus, we rewrite 
$$
Q_1(s)=\f{\sum_{j=1}^\ell q_j(x) (s^{\al_j}-s^{\wt\al_j})}{\sum_{j=1}^\ell \left|s^{\al_j}-s^{\wt\al_j}\right|} 
= \f{q_{j_0}(x) + \sum_{j=1}^{j_0 -1}q_j(x)\f{s^{\al_j}-s^{\wt\al_j}}{s^{\al_{j_0}}-s^{\wt\al_{j_0}}}}{\sum_{j=1}^{j_0}\f{|s^{\al_j}-s^{\wt\al_j}|}{|s^{\al_{j_0}}-s^{\wt\al_{j_0}}|}\!\cdot\! \f{|s^{\al_{j_0}}-s^{\wt\al_{j_0}}|}{s^{\al_{j_0}}-s^{\wt\al_{j_0}}}}
= \f{q_{j_0}(x) + \sum_{j=1}^{j_0-1}q_j(x)A_j(s)}{1+ \sum_{j=1}^{j_0-1}|A_j(s)|},
$$
where
$$
A_j(s) :=\f{s^{\al_j}-s^{\wt\al_j}}{s^{\al_{j_0}}-s^{\wt\al_{j_0}}}, \quad 1\le j\le j_0-1. 
$$
We now prove that $\lim_{s\to0+}A_j(s)=0$ for any $1\le j\le j_0-1$ which immediately implies \eqref{lim-Q}. In fact, there are the following three cases:

\noindent (\romannumeral1) $\al_j>\wt\al_j$ :
$$
\lim_{s\to0+} A_j(s) = \lim_{s\to0+} \f{s^{\al_j-\wt\al_j}-1}{s^{\al_{j_0}-\wt\al_j} - s^{\wt\al_{j_0}-\wt\al_j}} = \lim_{s\to0+} \f{s^{\al_j-\wt\al_j}-1}{s^{\wt\al_{j_0}-\wt\al_j}(s^{\al_{j_0}-\wt\al_{j_0}}-1)}.
$$
Since $\al_j>\wt\al_j$, $\wt\al_{j_0}<\wt\al_j$ ($1\le j\le j_0-1$) and $\al_{j_0}<\wt\al_{j_0}$, we have 
$$
\lim_{s\to0+} (s^{\al_j-\wt\al_j}-1) = -1,\quad \lim_{s\to0+} s^{\wt\al_{j_0}-\wt\al_j} = \infty, \quad \lim_{s\to0+}(s^{\al_{j_0}-\wt\al_{j_0}}-1) = \infty,
$$
which leads to $\lim_{s\to0+}A_j(s)=0$.
\vspace{0.2cm}

\noindent (\romannumeral2) $\al_j=\wt\al_j$ : trivial, $A_j(s)=0$ for any $s\in (0,1)$.
\vspace{0.2cm}

\noindent (\romannumeral3) $\al_j<\wt\al_j$ :
$$
\lim_{s\to0+} A_j(s) = \lim_{s\to0+} \f{s^{\wt\al_j}-s^{\al_j}}{s^{\wt\al_{j_0}}-s^{\al_{j_0}}} = \lim_{s\to0+} \f{s^{\wt\al_j-\al_j}-1}{s^{\al_{j_0}-\al_j}(s^{\wt\al_{j_0}-\al_{j_0}}-1)}.
$$
Since $\al_j<\wt\al_j$, $\al_{j_0}<\al_j$ and $\al_{j_0}<\wt\al_{j_0}$, we have 
$$
\lim_{s\to0+} (s^{\wt\al_j-\al_j}-1) = -1,\quad \lim_{s\to0+} s^{\al_{j_0}-\al_j} = \infty, \quad \lim_{s\to0+}(s^{\wt\al_{j_0}-\al_{j_0}}-1) = -1,
$$
which leads to $\lim_{s\to0+}A_j(s)=0$.

Next we consider the following problem
\begin{equation}
\label{equ-w0}
\left\{\begin{alignedat}{2}
& (\mathcal A-b)w_0(x)
=aq_{j_0}, 
&\quad& x\in\Om, \\
& w_0(x)=0, &\quad& x\in\pa\Om,
\end{alignedat}\right.
\end{equation}

We evaluate $w(s)-w_0$ by investigating the difference of the two systems \eqref{equ-w} and \eqref{equ-w0}, say, the following boundary value problem
\begin{equation}
\label{equ-ww0}
\left\{\begin{alignedat}{2}
&  (\mathcal A-b)(w(s) - w_0)
=-\sum_{j=1}^\ell q_j(x) s^{\al_j} w(s) - s\wh{\wt u}(s)Q_1(s) + a(Q_1(s)-q_{j_0}), 
&\quad& x\in\Om, \\
& w(s) - w_0 = 0, &\quad& x\in\pa\Om,
\end{alignedat}\right.
\end{equation}

Combined with \eqref{esti-u2<u0}, noting that $|Q_1(s)|\le \max_{1\le j\le \ell}\|q_j\|_{L^\infty(\Om)}$, we conclude that
$$
\|(\mathcal A-b)(w(s) - w_0)\|_\LL
\le C\sum_{j=1}^\ell s^{\al_j}\|w(s)\|_\LL + Cs^{\wt \al_\ell}\|a\|_\LL + C\|(Q_1(s)-q_{j_0}) a\|_\LL,
$$
for all $0<s<1$. Then by relations \eqref{esti-sw} and \eqref{lim-Q}, we obtain

$$
\lim_{s\to0+}\|(\mathcal A-b)(w(s) - w_0)\|_{L^2(\Om)}=0.
$$
Since $w(s)-w_0 = 0$ on $\pa\Om$, in terms of the regularity argument for elliptic equations, we have

$$
\|w(s) - w_0\|_{H^2(\Om)} \le C\|(\mathcal A-b)(w(s) - w_0)\|_{L^2(\Om)} + C\|w(s) - w_0\|_{L^2(\Om)}.
$$
In fact, the last term on the RHS vanishes thanks to the uniqueness of the solution to $(\mathcal A-b)u=f$ with Dirichlet boundary condition under $b\ge0$, which implies

$$
\lim_{s\to0+}\|w(s) - w_0\|_{H^2(\Om)}=0.
$$
From the Sobolev embedding theorem, we see that $H^2(\Om)$ can be embedded into $C(\ov\Om)$ in the case of $d=1,2,3$, and then
$$
\lim_{s\to0+} w(x;s) - w_0(x) = 0
$$
holds true for any $x\in\Om$. 
On the other hand, we conclude from the strong maximum principle for the elliptic equations that 
$$
\begin{cases}
w_0>0 \mbox{ in } \Omega & \mbox{if } a\ge0,\not\equiv 0,\\
w_0<0 \mbox{ in } \Omega & \mbox{if } a\le0,\not\equiv 0. 
\end{cases}
$$
Consequently
$$
\lim_{s\to0+} |w(x_0;s)|=|w_0(x_0)|>0
$$
where $x_0\in\Om$ is the fixed interior point in the theorem.
 
Therefore, we can choose $\ve_0>0$ such that $|w(x_0;s)|>0$ for any $s\in(0,\ve_0)$, which implies that 
$$
s|v(x_0;s)|>0,\quad s\in(0,\ve_0),
$$
hence that
$$
\wh u(x_0;s) \ne \wh{\wt u}(x_0;s),\quad s\in(0,\ve_0).
$$
This is a contradiction bacause $u(x_0,\cdot)=\wt u(x_0,\cdot)$ in $(0,T)$ combining the analyticity of $u,\wt u$ in $(0,\infty)$ implies that $\wh u(x_0;s)=\wh{\wt u}(x_0;s)$ for any $s>0$. By contradiction, we must have
$$
(\al_\ell,\dots,\al_1)=(\wt\al_\ell,\dots,\wt\al_1).
$$
Remind that we have assumed $\al_{j_0}<\wt\al_{j_0}$. However, in the opposite case, we just need to change the positions of $u$ and $\wt u$. 
This completes the proof of the theorem.
\end{proof}

\section{Concluding remarks}
\label{sec-rem}
In this paper, we considered the initial-boundary value problem for the multi-term time-fractional diffusion equations. For the forward problem,  by regarding the lower fractional order term as a perturbation for the highest fractional term and via the Mittag-Leffler function and the eigenfunction expansion argument, we gave an integral equation of the solution $u$ to the initial-boundary value problem, from which we verified the well-posedness of the initial-boundary value problem \eqref{equ-multifrac} including the unique existence and the analyticity of the solution by employing a general Gronwall inequality and fixed point method. Moreover, in the case where all the coefficients of the time-fractional derivatives are positive, by a Laplace argument, it turns out that the equation demonstrates a polynomial decay of its solution at infinity, and the decay rate of the solution for long time is dominated by $t^{-\alpha_\ell}$, which can be regarded as a generalization of the asymptotic behavior result in Li et al. \cite{LiLiuYa14} where they dealt with the case of positive-constant coefficients. For the inverse problem, on the basis of analyticity of the solution, we showed that the fractional orders can be uniquely determined from one interior point observation. Here we should mention that in the proofs of our results, we need the assumption that all the coefficients are only $x$-dependent. It will be more interesting and challenging to consider what happens with the properties of the solutions in the case where the coefficients are both $t$- and $x$- dependent.

\section*{Acknowledgement}
The first author thanks Grant-in-Aid for Research Activity Start-up 16H06712, JSPS. The second author thanks the Leading Graduate Course for Frontiers of Mathematical Sciences and Physics (FMSP, the University of Tokyo). This work was supported by A3 Foresight Program \lq\lq Modeling and Computation of Applied Inverse Problems\rq\rq, Japan Society of the Promotion of Science (JSPS). The third author is supported by Grant-in-Aid for Scientific Research (S) 15H05740, JSPS.


 \end{document}